\documentclass[11pt, reqno]{amsart}
\usepackage{latexsym,amsmath,amsthm,amsfonts, amssymb}
\usepackage{mathrsfs}
\usepackage{url}

\usepackage[active]{srcltx}\usepackage[colorlinks,pdfpagelabels,pdfstartview=FitH,bookmarksopen=true,bookmarksnumbered=true,linkcolor=blue,plainpages=false,hypertexnames=false,citecolor=red]{hyperref}

\usepackage{amsmath}
\usepackage{color}

\definecolor{blu}{rgb}{0,0,0.1}
\def\sideremark#1{\ifvmode\leavevmode\fi\vadjust{\vbox to0pt{\vss% the remark
 \hbox to 0pt{\hskip\hsize\hskip1em%                          will appear only
 \vbox{\hsize2.1cm\tiny\raggedright\pretolerance10000%          on the side
  \noindent #1\hfill}\hss}\vbox to15pt{\vfil}\vss}}}%

\usepackage{epsfig,color,xcolor}

\makeatletter
\def\itemize{
  \ifnum\@itemdepth>3\@toodeep\else
    \advance\@itemdepth\@ne
    \edef\@itemitem{labelitem\romannumeral\the\@itemdepth}%
        \list{\csname\@itemitem\endcsname}%
      {\leftmargin=20pt\def\makelabel##1{\hss\llap{##1}}}%%
        \fi}

\renewenvironment{enumerate}{%
  \ifnum \@enumdepth >3 \@toodeep\else
      \advance\@enumdepth \@ne
      \edef\@enumctr{enum\romannumeral\the\@enumdepth}\list
      {\csname label\@enumctr\endcsname}{\usecounter
        {\@enumctr}\leftmargin=20pt\def\makelabel##1{\hss\llap{\upshape##1}}}\fi
}{%
  \endlist
}

\makeatletter
\def\@settitle{\begin{center}%
  \baselineskip14\p@\relax
    \bfseries\@title
  \end{center}%
}
\def\@setauthors{%
  \begingroup
  \def\thanks{\protect\thanks@warning}%
  \trivlist
  \centering\footnotesize \@topsep30\p@\relax
  \advance\@topsep by -\baselineskip
  \item\relax
  \author@andify\authors
  \def\\{\protect\linebreak}%
  \authors%
  \ifx\@empty\contribs
  \else
    ,\penalty-3 \space \@setcontribs
    \@closetoccontribs
  \fi
  \endtrivlist
  \endgroup
}
\def\@setthanks{\def\thanks##1{\par##1}\thankses}
\makeatother

   \makeatletter
   \def\LaTeX{\leavevmode L\raise.42ex
       \hbox{\kern-.3em\size{\sf@size}{0pt}\selectfont A}\kern-.15em\TeX}
   \makeatother
   
   \newcommand{\BibTeX}{{\rm B\kern-.05em{\sc
             i\kern-.025emb}\kern-.08em\TeX}}
\def\bbm[#1]{\mbox{\boldmath $#1$}}



   \newcommand{\e }{\varepsilon }

   \newcommand{\R}{{\mathbb{R}}}
   
   \newcommand{\Z}{\mathbb{Z}}
   
\newcommand{\cal}{\mathcal }  
   
   \newcommand{\N}{\mathbb{N}}
   \newcommand{\beq}{\begin{equation}}
   \newcommand{\eeq}{\end{equation}}

   \newtheorem{theorem}{Theorem}[section]
   \newtheorem{definition}[theorem]{Definition}
   \newtheorem{proposition}[theorem]{Proposition}
   \newtheorem{lemma}[theorem]{Lemma}
   \newtheorem{corollary}[theorem]{Corollary}
   \newtheorem{remark}[theorem]{Remark}
     \newtheorem{example}[theorem]{Example}
   \newcommand{\bremark}{\begin{remark} \em}
   \newcommand{\eremark}{\end{remark} }

\def\bbm[#1]{\mbox{\boldmath $#1$}}
\def\bbm[#1]{\mbox{\boldmath $#1$}}

\begin{document}

\title[Prescribed Gauss curvature problem]
{\LARGE{Prescribed Gauss curvature problem \\ on  singular surfaces}}
\author[T. D'Aprile \& Francesca De Marchis \& Isabella Ianni]{\large \sc Teresa D'Aprile \and Francesca De Marchis \and Isabella Ianni}
\address{Teresa D'Aprile, Dipartimento di Matematica, Universit\`a di Roma ``Tor
Vergata", via della Ricerca Scientifica 1, 00133 Roma, Italy.}
\email{daprile@mat.uniroma2.it}
\address{Francesca De Marchis,  Dipartimento di Matematica, Universit\`a di Roma
``Sapienza", piazzale Aldo Moro 5, 00185 Roma, Italy.
}
\email{demarchis@mat.uniroma1.it}
\address{Isabella Ianni, Dipartimento di Matematica e Fisica,  Universit\`a degli Studi della Campania ``Luigi Vanvitelli", Viale Lincoln 5, 81100 Caserta, Italy.}
\email{isabella.ianni@unina2.it}

\begin{abstract} We study the existence of at least one  conformal metric of prescribed  Gaussian curvature on a closed surface $\Sigma$ admitting conical singularities of orders $\alpha_i$'s at points $p_i$'s. In particular, we are concerned with the case where the prescribed Gaussian curvature is sign-changing. Such a geometrical problem reduces to solving  a singular Liouville equation. By employing a min-max scheme jointly with a finite dimensional reduction method,   we  deduce new perturbative results providing existence     when the quantity
$\chi(\Sigma)+\sum_i \alpha_i$ approaches a positive even integer, where  $\chi(\Sigma)$ is the Euler characteristic of the surface $\Sigma$. 

\medskip



\end{abstract}
\maketitle
\section{Introduction} 
Let $(\Sigma, g)$ be a compact orientable surface without boundary endowed with metric $g$ and Gauss curvature $\kappa_g$.  Given a Lipschitz function $K$ defined on $\Sigma$, a classical problem in differential geometry is the question on the existence of a metric $\tilde g$ on $\Sigma$ conformal to $g$: $$\tilde g = e^{u} g$$ (with $u$ a smooth function on $\Sigma$) of prescribed Gauss curvature $K$. In particular,  in the case of constant function $K$ the above question is referred to as classical \textit{Uniformization problem}, whereas for general function this is known as the \textit{Kazdan-Warner problem} (or the \textit{Nirenberg problem} in the case of the standard sphere). The problem of finding a conformal metric of prescribed
Gauss curvature $K$  amounts to solving the equation \beq\label{reg}-\Delta_g u+2\kappa_g =2K e^u.\eeq Here $\Delta_g$
is the Laplace-Beltrami operator.  The solvability of this problem  so far has not been completely
settled, aside from the case of surfaces with zero Euler characteristic (\cite{KaWa}). In particular, both in the case  of a topological sphere and in the case when $\Sigma$ has negative Euler characteristic only partial results are known (\cite{Au}, \cite{Be}, \cite{BoGa}, \cite{ChYa}, \cite{ChYa2}, \cite{ChLi}, \cite{ChLi2},  \cite{ChLi3}, \cite{DeRo}).

In this paper we will focus on a singular version of the problem \eqref{reg}.
Following the pioneer work of Troyanov \cite{Troy}, we say that $(\Sigma,\tilde{g})$ defines a punctured Riemann surface $\Sigma\setminus\{p_1,\ldots,p_m\}$ that admits a conical singularity of order $\alpha_i>-1$ at the point $p_i$, for any $i=1,\ldots,m$, if in a coordinate system 
$z=z(p)$ around $p_i$ with $z(p_i)=0$ we have $$\tilde g (z)=|z|^{2\alpha_i}e^{w} |dz|^2$$ with $w$ a smooth function. In other words,  $\Sigma$ admits a tangent cone with vertex at $p_i$ and total angle $\theta_i=2\pi(1+\alpha_i)$ for any $i$. The Gauss curvature at any vertex is a Dirac mass with magnitude $-2\pi \alpha_i$. Clearly we can assume $\alpha_i\neq0$, indeed for the \emph{round angle} $\theta_i=2\pi$, corresponding to $\alpha_i=0$, we would have no singular part either.

For a given Lipschitz function $K$ defined on $\Sigma$, we address the question to find a metric $\tilde g$ conformal to $g$ in $\Sigma\setminus\{p_1,\ldots, p_m\}$, namely $$\tilde g = e^{u} g\quad\hbox{ in }\Sigma\setminus\{p_1,\ldots, p_m\}$$ (with $u$ a smooth function on the punctured surface), admitting conical singularities of orders $\alpha_i$'s at the points $p_i$'s and having $K$ as the associated Gaussian curvature in $\Sigma\setminus\{p_1,\ldots,p_m\}$. Similarly to the regular case \eqref{reg}, the question reduces to solving a singular Lioville-type equation on $\Sigma$:
\begin{equation}\label{sing}
-\Delta_g u+2\kappa_g=2K e^u-4\pi\sum_{i=1}^m \alpha_i \delta_{p_i}\qquad\mbox{in $\Sigma$}.
\end{equation}

A first information is given by the Gauss-Bonnet formula: indeed, integrating \eqref{sing} one immediately obtains 
\begin{equation}\label{GBsing}
2\int_{\Sigma} K e^u dV_g=2\int_\Sigma \kappa_gdV_g+4\pi\sum_{i=1}^m \alpha_i=4\pi \Big(\chi(\Sigma)+\sum_{i=1}^m \alpha_i\Big),
\end{equation}
where $dV_g$ denotes the area element in $(\Sigma,g)$ and $\chi(\Sigma)$ is the Euler characteristic of the surface. Analogously to what happens for the regular case, the solvability of \eqref{sing} depends crucially on the value of the generalized Euler characteristic for singular surfaces defined as follows
\beq\label{geneul}
\chi(\Sigma,\underline{\alpha})=\chi(\Sigma)+\sum_{i=1}^m\alpha_i.
\eeq
When $\chi(\Sigma,\underline{\alpha})\leq0$ Troyanov \cite{Troy} obtained existence results analogous to the ones for the regular case (\cite{Be}, \cite{KaWa}).\\
Whereas if $\chi(\Sigma,\underline{\alpha})>0$, then \eqref{GBsing} implies that the function $K$ has to be positive somewhere to allow the solvability of \eqref{sing}. In \cite{Troy} it is proved that if $\chi(\Sigma,\underline{\alpha})\in(0,2(1+\min\{0,\alpha_1,\ldots,\alpha_m\}))$, this necessary condition is also sufficient to guarantee existence of a solution.

\

Let us transform equation \eqref{sing} into another one which admits a variational structure. Let $G(x,p)$ be the Green's function of $-\Delta_g$ over $\Sigma$ with singularity at $p$, namely $G$ satisfies
\begin{equation*}
\left\{
\begin{aligned}
&-\Delta_gG(x,p)=\delta_p-\frac{1}{|\Sigma|} &\hbox{ on }&\Sigma\\
&\int_{\Sigma} G(x,p)dV_g=0\end{aligned}\right.
\end{equation*}where $|\Sigma|$ is the area of $\Sigma$, that is $|\Sigma|=\int_\Sigma dV_g$.
Next, having $\frac{4\pi\chi(\Sigma)}{|\Sigma|}-2\kappa_g(x)$ zero mean value,   we define $f_g$ to be the (unique) solution of
\begin{equation}\label{delgel}
\left\{\begin{aligned}
&-\Delta_g f_g(x)=\frac{4\pi\chi(\Sigma)}{|\Sigma|}-2\kappa_g(x)  &\hbox{ on }&\Sigma\\
&\int_{\Sigma} f_g(x)dV_g=0.\end{aligned}\right.
\end{equation} 
By the change of variable  
\begin{equation}\label{v}
v= u+4\pi \sum_{i=1}^m\alpha_i G(x, p_i)-f_g,
\end{equation}
problem \eqref{sing} is then equivalent to solving the following (regular)  problem
\beq\label{liouvgeom} \tag*{$(*)_{\rho_{geo}}$}
\left\{\begin{aligned}
&-\Delta_g v=\rho_{geo}\bigg(\frac{\tilde K(x)e^{v}}{\int_{\Sigma} \tilde K(x) e^v\,dV_g}-\frac{1}{|\Sigma|}\bigg)\quad\hbox{ on }\Sigma
\\&\rho_{geo}=   4\pi\chi(\Sigma,\underline{\alpha})\end{aligned}\right.\eeq
where $\tilde K(x)$ is the function \beq\label{aaa}\tilde K(x)=K(x)e^{f_g(x)-4\pi \sum_{i=1}^m \alpha_iG(x,p_i)}.\eeq
Notice that, since $G(x,p)$ can be decomposed as 
\begin{equation} \label{0948}
G(x,p)=\frac{1}{2\pi} \log \frac{1}{d_g(x,p)}+h(x,p)\qquad h \in {C}^1(\Sigma^2),
\end{equation} 
where $d_g$ is the distance induced on $\Sigma$ by $g$, we have
\beq
\label{asym} \tilde K(x)\simeq K(x)d_g(x,p_i)^{2\alpha_i} e^{\gamma_i(x)}\quad\hbox{ for $x$ close to $p_i$}
\eeq
for some functions $\gamma_i\in C^1(\Sigma)$.

It is worth to observe that more  generally one could replace the function $f_g$ appearing in \eqref{v} and \eqref{aaa} by any regular function $a_g$ having zero mean value, obtaining (with minor changes) analogous results, but  for the sake of simplicity we will not comment on this issue any further.

\

A possible strategy to solve problem \ref{liouvgeom} is to study the following Liouville problem
\beq\label{liouvhat}\tag*{$(*)_\rho$}  -\Delta_g v=\rho\bigg(\frac{\tilde K(x)e^{v}}{\int_{\Sigma} \tilde K(x) e^v\,dV_g}-\frac{1}{|\Sigma|}\bigg)\quad\hbox{ on }\Sigma\eeq
 for $\rho$ positive independent of $\Sigma$ and $\alpha_i$, and to deduce a posteriori the answer to the geometric question taking $\rho=\rho_{geo}$. 
%\bigskip
Since problem \ref{liouvhat} has a variational structure, its solutions can be found as critical points of the associated energy  functional 
$$J_\rho(v):=\frac12\int_{\Sigma} |\nabla_g v|^2dV_g+\frac{\rho}{|\Sigma|}\int_{\Sigma} v \,dV_g-\rho\log\bigg(\int_{\Sigma} \tilde K(x)e^v dV_g\bigg),$$
defined in the domain
$$X=\bigg\{v\in H^1(\Sigma)\,\bigg|\,\int_{\Sigma} \tilde K (x)e^v dV_g>0\bigg\}.$$

Problem \ref{liouvgeom} has been widely investigated in literature in the case $\chi(\Sigma,\underline{\alpha})>0$ when $K$ is a strictly positive function and even more results are available on \ref{liouvhat} for $\rho>0$ when $K$ is positive, which is a relevant question also from the physical point of view, see for example \cite{Tarantello} and the references therein.

In \cite{BT} (see also \cite{BMont}), under the hypotheses $K>0$,
it is shown that a sequence $u_{\rho_n}$ of solutions to $(*)_{\rho_n}$ may blow up only if $\rho_n\to\rho$ with $\rho$ belonging to the following discrete set of values
\begin{equation}\label{Gamma}
\Gamma(\underline{\alpha}_m)=\left\{8\pi n+8\pi\sum_{i\in I}(1+\alpha_i)\,\bigg|\,n\in\N\cup\{0\},\,I\subset\{1,\ldots,m\}\right\}.
\end{equation}
Using this compactness result, in \cite{BDM} it is proved via a Morse theoretical approach that if $\alpha_i>0$ and 
$\chi(\Sigma)\leq0$ then \ref{liouvhat} is solvable for all $\rho\notin \Gamma(\underline{\alpha}_m)$. In the case of surfaces with positive Euler characteristic (which is the most delicate), under some extra hypotheses on the $\alpha_i$'s, in \cite{MalchiodiRuizSphere} the solvability of $(*)_{\rho}$ for $\rho\in(8\pi,16\pi)\setminus \Gamma(\underline{\alpha}_m)$ is established. Still for $K$ strictly positive, the case when the $\alpha_i's$ are negative has been considered in \cite{Carl} and \cite{CarlMal}.

The special case of prescribing positive constant curvature on $\mathbb{S}^2$ with $m=2$ is considered first in \cite{Troy1}, where it is shown that \ref{liouvgeom} admits a solution only if $\alpha_1=\alpha_2$ and this implies (taking $\alpha_2=0$) that no solution exists 
for $m=1$. Furthermore, necessary and sufficient conditions on the $\alpha_i$'s for the solvability with $m=3$ are determined in \cite{Eremenko}.

More recently, in \cite{ChenLin} the Leray-Schauder degree of \ref{liouvhat} has been computed for $\rho\notin\Gamma(\underline{\alpha}_m)$, recovering some of the previous existence results and obtaining new ones in the case $\chi(\Sigma)>0$. Anyway on the sphere there are still different situations in which the degree vanishes and the solvability is an open problem.

\

All the above results are concerned with the case $K>0$. Up to our knowledge the singular problem \ref{liouvgeom} with $K$ sign-changing has been considered only in \cite{fra&raf} when the surface is the standard sphere $(\mathbb{S}^2,g_0)$ and in \cite{fra&raf&dav} for a general surface under mild assumptions on the nodal set of $K$ (see Remarks \ref{rem:counterpart1} and \ref{rem:counterpart2}). 

\

In this paper we will mainly consider the problem \ref{liouvgeom} with $K$ sign-changing, obtaining new existence results via a perturbative approach already applied in \cite{tea&pp} to deal with \ref{liouvhat} in the case $K$ positive, and in \cite{dap} and \cite{DKM}   for the corresponding Liouville-type equation in a Euclidean context.

\

We define the set $$\Sigma^+:=\left\{\xi\in \Sigma\,\Big|\, K(\xi)>0\right\},$$ 
and in order to state our results we introduce the following hypotheses on $K$, on the $p_i$'s and the $\alpha_i$'s:

\begin{enumerate}
\item[(H1)] $K$ sign-changing, namely $K(\xi)K(\eta) < 0$ for some $\xi,\eta\in \Sigma$;
\item[(H2)] $K\in {C}^{2}(\Sigma)$;
\item[(H3)] $\nabla K(\xi)\neq 0$ for all $\xi\in \partial \Sigma^+$;
\item[(H4)] $p_i\in\Sigma\setminus\partial \Sigma^+$ for all $i\in\{1,\ldots,m\}.$
\end{enumerate}

In virtue of (H4) we may assume, up to reordering, that
\begin{equation}\label{ell}
p_i\in \Sigma^+ \hbox{ for }i\in\{1,\ldots, \ell\}\hbox{ and } p_i\in \Sigma\setminus \overline{\Sigma^+}\hbox{ for }i\in\{\ell+1,\ldots, m\}
\end{equation} 
for some $0\leq\ell\leq m$. 
We are now ready to present our main perturbative results which provides existence     when  the quantity
$\sum_{i=1}^m \alpha_i+\chi(\Sigma)$ approaches an even integer from the left hand side.

\begin{theorem}\label{thm:intro N+} 
Let $N\in\N$. Assume that $\Sigma^+$ has $N^+$ connected components with $N^+\geq N$, hypotheses (H1),(H2) hold and
$$\Delta_g (\log K(x))\leq -\beta<0\quad \forall x\in\Sigma^+$$ for some $\beta>0$.
Then for any $\alpha_\star>-1$  there exists $\delta\in(0,\beta|\Sigma|)$ such that if $\alpha_1,\ldots,\alpha_m>\alpha_\star$  satisfy $$\alpha_i>0\quad \forall i=1,\ldots,\ell,$$ 
$$\sum_{i=1}^m \alpha_i=2N-\chi(\Sigma)-\frac{\varepsilon}{4\pi}\quad \mbox{ for some } \varepsilon\in (0,\delta),$$
 then \ref{liouvgeom} admits a solution $v_\e$ with $\rho_{geo}=8\pi N-\varepsilon, $ i.e., $K$ is the Gaussian curvature of at least one metric conformal to $g$ and having a conical singularity at $p_i$ with order $\alpha_i$.
  Moreover there exist distinct points $\xi_1^*, \ldots, \xi_N^*\in \Sigma^+\setminus \{p_1,\ldots,p_\ell\}$ such that 
\beq\label{conv0}\rho\frac{\tilde K(x)e^{v_\e}}{\int_{\Sigma}\tilde K(x)e^{v_\e}dV_g}\to 8\pi \sum_{j=1}^N\delta_{\xi_j^\ast}\quad\hbox{ as }\e\to 0^+
\eeq in the measure sense.
\end{theorem}

\begin{theorem}\label{thm:intro noncontractible}
Let $N\in\N$. Assume that $\Sigma^+$ has a non contractible connected component, hypotheses (H1), (H2), (H3), (H4) hold and
$$\Delta_g (\log K(x))\leq -\beta<0\quad \forall x\in\Sigma^+$$ for some $\beta>0$. Then for any $\alpha_\star>-1$  there exists $\delta\in(0,\beta|\Sigma|)$ such that if $\alpha_1,\ldots,\alpha_m>\alpha_\star$ satisfy $$\alpha_i \neq 0,1,2,\dots, N-1\quad \forall  i=1,\dots, \ell,$$    
$$\sum_{i=1}^m \alpha_i=2N-\chi(\Sigma)-\frac{\varepsilon}{4\pi}\quad \mbox{ for some } \varepsilon\in (0,\delta),$$
then \ref{liouvgeom} admits a solution with 
$\rho_{geo}=8\pi N-\varepsilon$. Moreover there exist distinct points $\xi_1^*, \ldots, \xi_N^*\in \Sigma^+\setminus \{p_1,\ldots,p_\ell\}$ such that \eqref{conv0} holds.
\end{theorem}

\begin{remark}\label{rem:counterpart1}
The previous two results are a sort of perturbative counterpart of the global existence results established in \cite[Theorem 1.2]{fra&raf} if $(\Sigma, g)=(\mathbb{S}^2,g_0)$ and in \cite[Theorem 2.2]{fra&raf&dav} for a general surface. Indeed, in  \cite[Theorem 1.2]{fra&raf} and in  \cite[Theorem 2.2]{fra&raf&dav}   it has been shown that if $\rho_{geo}\notin \Gamma(\underline{\alpha}_\ell)$  (where $\ell$ is defined in \eqref{ell}) and the positive nodal region of $K$ has a non contractible connected component or a sufficiently large number of connected components (precisely, a number of connected components  greater than $\frac{\rho_{geo}}{8\pi}$) then \ref{liouvhat} admits a solution. Nevertheless, in \cite{fra&raf}-\cite{fra&raf&dav} the behaviour of the solutions as $\rho\to 8\pi N^-$ is unknown, whereas the solutions constructed in  Theorems \ref{thm:intro N+} and \ref{thm:intro noncontractible} exhibit a blow-up phenomena, a property that has a definite interest in its own. \end{remark}

 When all the connected components of $\Sigma^+$ are simply connected or their number is not
sufficiently  large, then the solvability issue is more delicate and it is treated in the following theorem.
 \\
Hereafter for any $\alpha>-1$, the square bracket $[\alpha]$ stands for the integer part and 
$$[\alpha]^-:=\lim_{\e\to 0^+}[\alpha -\e]=\max\{n\,|\, n\in \Z,\; n<\alpha\}.$$

\begin{theorem}\label{thm:intro contractible}
Let $N\in\N$. Assume that hypotheses (H1), (H2), (H3), (H4) hold and $$\Delta_g (\log K(x))\leq -\beta<0\quad \forall x\in\Sigma^+$$ for some $\beta>0$.
Then for any $\alpha_\star>-1$  there exists $\delta\in(0,\beta|\Sigma|)$ such that if $\alpha_1,\ldots,\alpha_m>\alpha_\star$  satisfy $$\alpha_i \neq 0,1,2,\dots, N-1\quad \forall  i=1,\dots, \ell,$$ \begin{equation}\label{ineq intro}
N\leq\ell+\sum_{i=1}^\ell[\alpha_i]^-,
\end{equation} 
\beq\label{AA}\sum_{i=1}^m \alpha_i=2N-\chi(\Sigma)-\frac{\varepsilon}{4\pi}\quad \mbox{ for some } \varepsilon\in (0,\delta),\eeq then \ref{liouvgeom} admits a solution with 
$\rho_{geo}=8\pi N-\varepsilon$. Moreover there exist distinct points $\xi_1^*, \ldots, \xi_N^*\in \Sigma^+\setminus \{p_1,\ldots,p_\ell\}$ such that \eqref{conv0} holds.
\end{theorem}

\begin{remark}\label{rem:compatibility}
Let us observe  that the inequality \eqref{ineq intro} is consistent with the condition \eqref{AA}  provided that 
$$\sum_{i=1}^m\alpha_i<2\ell-\chi(\Sigma)+2\sum_{i=1}^\ell[\alpha_i]^-.$$
Roughly speaking, this requires that the total multiplicity $\sum_{i=1}^m\alpha_i$ has to be controlled by the first $\ell$ orders $\alpha_i$.
\end{remark}

\begin{remark}\label{rem:counterpart2} In \cite{fra&raf} and \cite{fra&raf&dav} also the case when $\Sigma^+$ has only contractible connected components  is addressed, deriving both existence results (under extra assumptions on $\rho_{geo}$, on the $\alpha_i$'s and on the location of the $p_i$'s) and non existence results: in particular the condition  $\rho_{geo}<8\pi \max_{i=1,\ldots,\ell}(1+\alpha_i)$ is required to get existence. By  Theorem \ref{thm:intro contractible} we get new existence results  for any $(\Sigma,g)$: indeed  if $\Sigma^+$ is contractible and the following conditions hold: $$\begin{aligned}&\alpha_1,\ldots,\alpha_\ell\in(0,1],\quad\rho_{geo}\in(8\pi,16\pi),\quad\rho_{geo}\geq 8\pi (1+\alpha_i)\;\;\forall i=1,\ldots,\ell,\end{aligned}$$
 then  the variational approach of \cite{fra&raf} and \cite{fra&raf&dav} breaks down not for technical reasons but being the low sublevels of the Euler Lagrange functional contractible. On the other hand Theorem \ref{thm:intro contractible} allows to produce a wide class of examples in which \ref{liouvgeom} admits a solution even in such situations  (see for instance Example \ref{example:cos}). In particular this provides existence in a perturbative regime allowing larger values of  $\rho_{geo}$ with respect to the papers \cite{fra&raf} and \cite{fra&raf&dav}.
\end{remark}
 
\begin{example}[Existence results for $\rho_{geo}\in(16\pi-\delta,16\pi)$]\label{example:cos}
If $K$ verifies (H1), (H2), (H3), (H4) and $\Sigma^+$ is contractible (consider for example on $(\mathbb{S}^2,g_0)$ the function $K(\phi)=\cos(\phi)$, defined in spherical coordinates, where $\phi$ is the polar angle), then, via Theorem \ref{thm:intro contractible} (with $N=2$), we can perform many configurations  $$p_1,\ldots,p_m\in\Sigma\setminus\partial\Sigma^+,\quad\alpha_1,\ldots,\alpha_m \hbox{ (even with }\alpha_1,\ldots,\alpha_\ell\in(0,1]), \;\;m\geq\ell\geq2$$ such that $$\rho_{geo}\geq 8\pi(1+\alpha_i)\quad \forall i=1,\ldots,\ell,\quad\rho_{geo}\in(16\pi-\delta,16\pi)$$ for a sufficiently small $\delta>0$ and \ref{liouvgeom} admits a solution (see Figure \ref{figureSfera} below). For instance, the case when $m\geq \ell\geq 2$ and $\alpha_i=\alpha$ for all $i=1,\ldots, m$ with $\alpha$ in a small left neighborhood  of $\frac{2}{m}$   satisfies the above conditions together with \eqref{ineq intro} and \eqref{AA}, so solvability is assured   by Theorem \ref{thm:intro contractible}. 
\\ It is worth to notice that none of the situations described above was covered by the results in \cite{fra&raf}.
\begin{figure}[h]
  \centering
  \def\svgwidth{350pt}
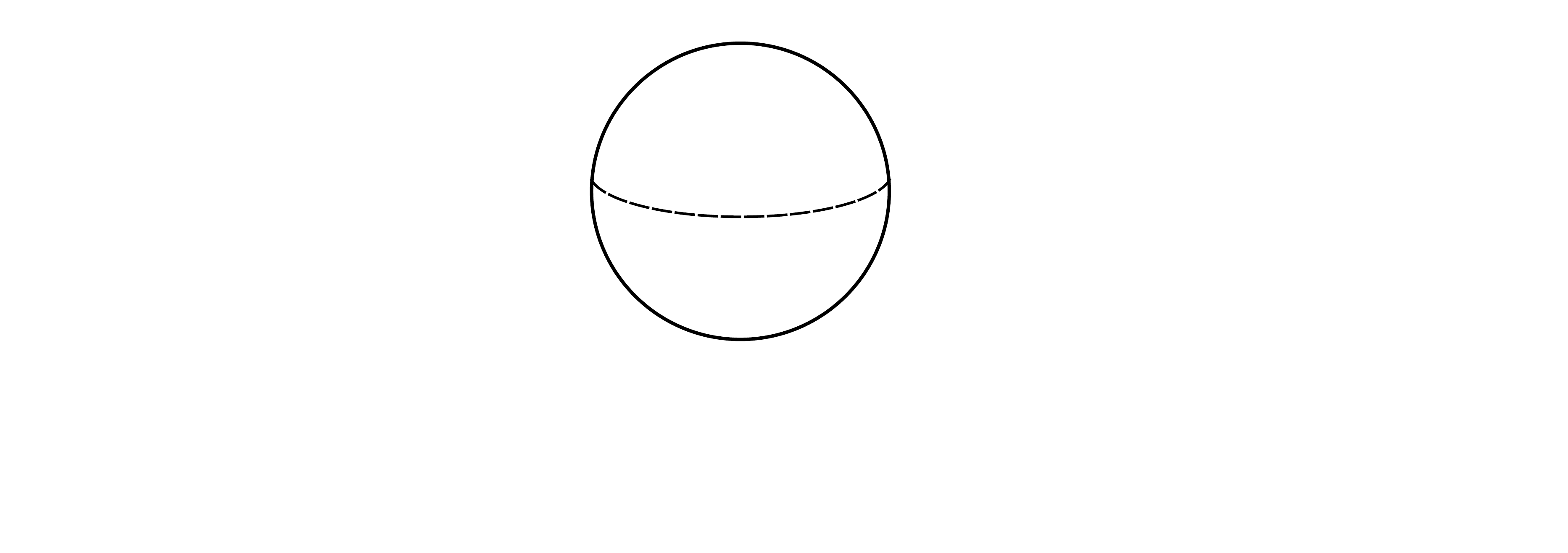
  \caption{$(\mathbb{S}^2,g_0)$, $\Sigma^+$ contractible, $\rho_{geo}= 16\pi-\varepsilon$}
\label{figureSfera}
\end{figure}
\end{example}

\bigskip

%\newpage

\begin{figure}[h]
\bigskip
\bigskip
  \centering
  \def\svgwidth{400pt} 
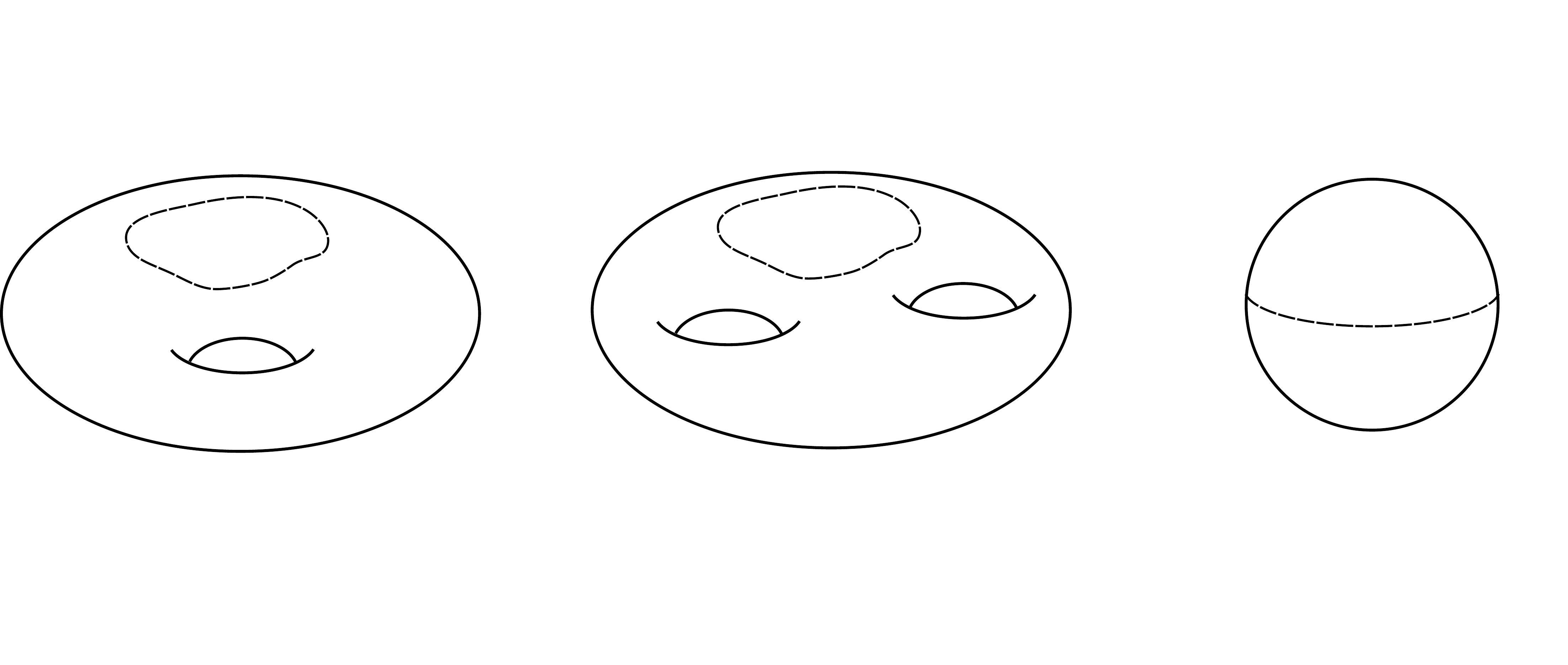
\end{figure}

\bigskip

\begin{figure}[h]
  \centering
\def\svgwidth{400pt}
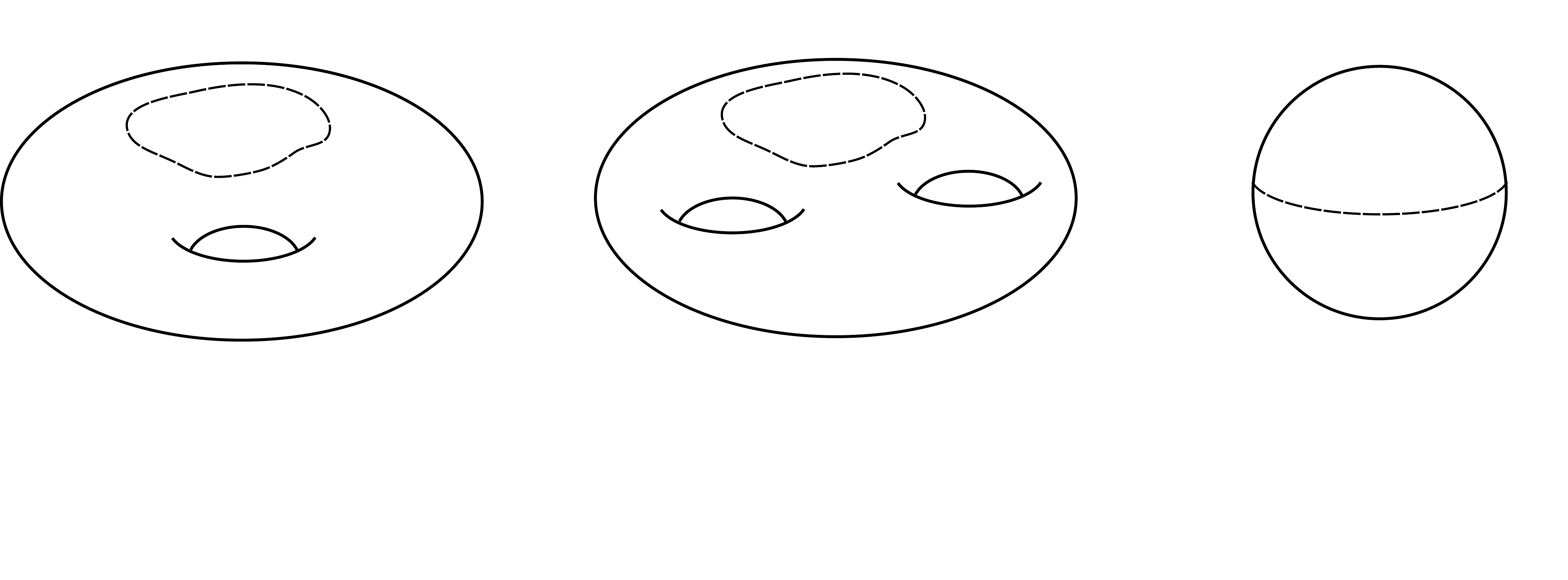
\end{figure}

\bigskip

\begin{figure}[h]
  \centering
\def\svgwidth{400pt}
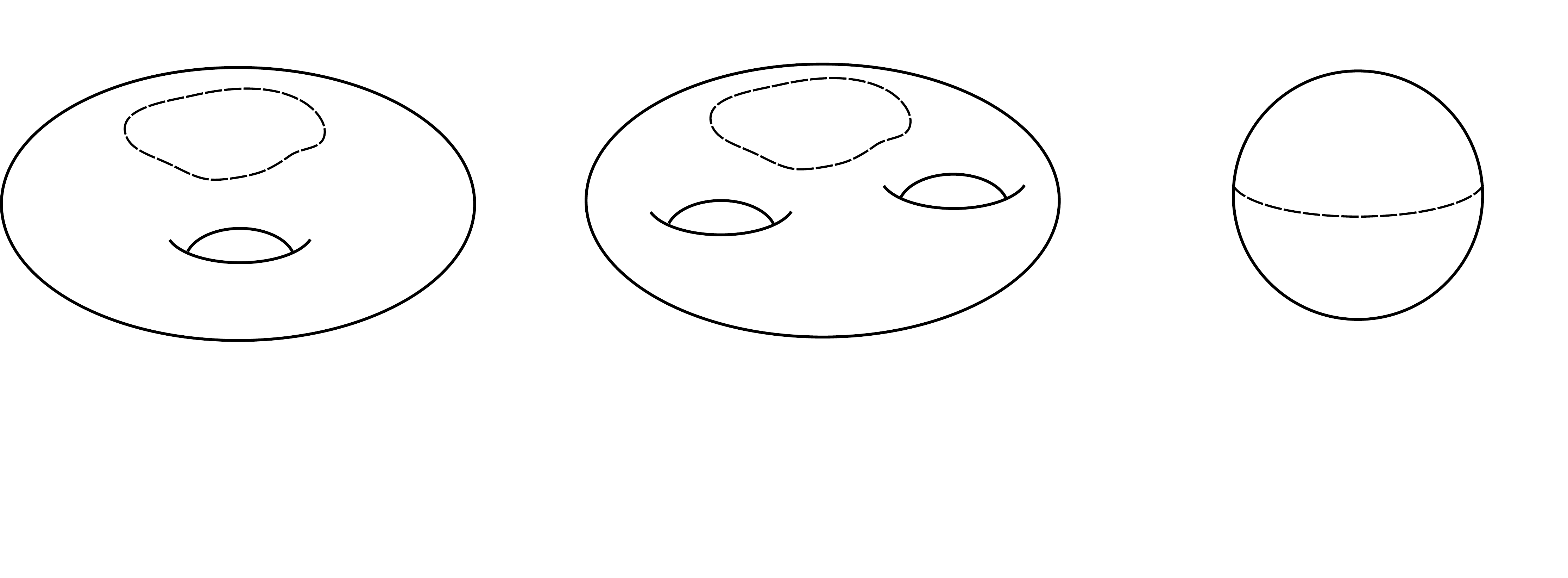
\bigskip
\medskip
  \caption{$\Sigma^+\mbox{ contractible}$, $\rho_{geo}=8\pi N-\varepsilon$, $N\geq 3$}
\label{figure:3examples}
\end{figure}

\begin{example}[Existence results for $\rho_{geo}>16\pi$]\label{example:rho>16pi}
Still by Theorem \ref{thm:intro contractible}, it is also possible to derive a wide class of existence results for $K$ satisfying (H1), (H2), (H3), (H4) with $\Sigma^+$ contractible and $\rho_{geo}>16\pi$ and this is completely new. We just present some concrete examples in Figure \ref{figure:3examples} above (where $K$ is assumed to be a fixed function satisfying the assumptions in Theorem \ref{thm:intro contractible}).
\end{example}

At last, as a direct byproduct of the perturbative approach already applied in \cite{espofigue} to deal with the Liouville equation \ref{liouvhat},   we can provide class of functions $K$ (positive or sign-changing) for which \ref{liouvgeom} is solvable, even in cases in which general existence results are not available. In particular we can also deal with situations when the degree of the equation (computed in \cite{ChenLin}) is zero so that solvability is not known in general, or when there are examples of $K$ for which \ref{liouvgeom} does not admit solutions.

We recall, for instance, that on $(\mathbb{S}^2,g_0)$, if $m=1$ and $\alpha_1>0$:
\begin{itemize}
\item it is proved in \cite{Troy} that \ref{liouvgeom} is not solvable with $K\equiv1$, namely the \emph{tear drop} conical singularity on $\mathbb{S}^2$ does not admit constant curvature (see also \cite{BLT} for a more general non existence result);
\item in the sign-changing case, if $\ell=0$, in \cite{fra&raf} it is shown that for a class of axially symmetric functions $K$, satisfying (H1), (H2), (H3) and (H4) and such that $\Sigma^+$ is contractible, equation \ref{liouvgeom} is not solvable.
\end{itemize}
Whereas on $(\mathbb{S}^2,g_0)$, if $m=3$, $\alpha_1=\alpha_2\in(-\tfrac13,0)$, $\alpha_3>2$: 
\begin{itemize}
\item according to the formula in \cite{ChenLin}, if $K$ is positive, the Leray-Schauder degree of the equation \ref{liouvhat} vanishes for $\rho\in(16\pi,8\pi(3+2\alpha_1))$.
\end{itemize}
Here considering functions $K$ having \emph{sufficiently convex} local minima or \emph{sufficiently concave} local maxima, as a counterpart of the above three non existence statements  we can prove the following three existence result Theorem  \ref{thm:examples1}, Theorem \ref{thm:examples2} and Theorem \ref{thm:examples3}, see Section \ref{Section:ultima} for further details and further examples.

\begin{theorem}
\label{thm:examples1} 
On the standard sphere $(\mathbb{S}^2,g_0)$ with  $m=1$ the following holds:
\begin{itemize}
\item[\emph{(i)}]  for any $N\in\N$ there exists a class of positive functions $K$  such that if 
\[(0<)\ \alpha_1=2(N-1) + \frac{\varepsilon}{4\pi}\quad \mbox{ for $\varepsilon>0$ small enough,}\] then \ref{liouvgeom} admits a solution with $\rho_{geo}=8\pi N + \varepsilon$; 
\item[\emph{(ii)}]  for any $N\in\N$, $N\geq 2$, there exists a class of positive functions $K$  such that if 
\[
(0<)\ \alpha_1=2(N-1) - \frac{\varepsilon}{4\pi}\quad \mbox{ for  $\varepsilon>0$ small enough,}
\] then \ref{liouvgeom} admits a solution with $\rho_{geo}=8\pi N - \varepsilon$; 
\end{itemize}
\end{theorem}

\begin{theorem}
\label{thm:examples2}
On the standard sphere $(\mathbb{S}^2,g_0)$ with  $m=1$ and $\ell=0$ the following holds:
\begin{itemize}
\item[\emph{(i)}]  for any $N\in\N$ there exists a class of functions $K$, satisfying (H1), (H2), (H3), (H4) and with $(\mathbb{S}^2)^+$ contractible, such that if 
\[
(0<)\ \alpha_1=2(N-1) + \frac{\varepsilon}{4\pi}\quad  \mbox{ for $\varepsilon>0$ small enough,}
\] then \ref{liouvgeom} admits a solution with $\rho_{geo}=8\pi N + \varepsilon$; 
\item[\emph{(ii)}] for any $N\in\N$, $N\geq2$, there exists a class of functions $K$, satisfying (H1), (H2), (H3), (H4) and with $(\mathbb{S}^2)^+$ contractible, such that if 
\[ (0<)\ \alpha_1=2(N-1) - \frac{\varepsilon}{4\pi}\quad \mbox{ for  $\varepsilon>0$ small enough,}
\] then \ref{liouvgeom} admits a solution with $\rho_{geo}=8\pi N - \varepsilon$;  
\end{itemize} 
\end{theorem}

\begin{theorem} 
\label{thm:examples3} On the standard sphere $(\mathbb{S}^2,g_0)$ with  $m=3$
there exists a class of positive functions $K$ such that if
\[
\left\{
\begin{array}{lr}
\alpha_1=\alpha_2\in (-\tfrac13,0)\\
\alpha_3=2-2\alpha_1+\frac{\varepsilon}{4\pi}\quad \mbox{ for some $\varepsilon>0$ small enough,}
\end{array}
\right.
\]
than \ref{liouvgeom} admits a solution with $\rho_{geo}=16\pi+\varepsilon$.
\end{theorem}

\

The paper is organized as follows. In Section \ref{section:fdr} we recall the finite-dimensional reduction developed  in \cite{espofigue} for the equation \ref{liouvhat}, which 
is the starting point of our analysis.  In particular in the reduction procedure the crucial role of stable critical points of the reduced energy arises  in the existence of solutions for \ref{liouvhat}. Then we state three general existence results for such  critical
points, which are contained in our Propositions \ref{prop:ptocrit0}, \ref{prop:ptocrit1}, \ref{prop:ptocrit2}.   In Section \ref{section:liouvhat}  we employ the reduction approach in order to  derive solutions for the more general equation \ref{liouvhat}, by which we deduce   Theorem \ref{thm:intro N+}, \ref{thm:intro noncontractible}, \ref{thm:intro contractible} as corollaries. Section \ref{section:minmax} is devoted to the proof of Propositions \ref{prop:ptocrit1}, \ref{prop:ptocrit2} (the proof of Proposition \ref{prop:ptocrit0} is instead immediate), which are at the core of this paper, by carrying out
a min-max scheme.   
At last  in Section \ref{Section:ultima} we focus on the problem \ref{liouvgeom} and we  provide several examples of solvability. 

\tableofcontents

\section{The finite dimension problem}\label{section:fdr}

The starting point for the proofs of Theorems \ref{thm:intro N+}, \ref{thm:intro noncontractible} and \ref{thm:intro contractible} is the \textit{finite dimension variational reduction} which has been carried out for the equation \ref{liouvhat} in the paper \cite{espofigue}, and 
 reduces the problem of finding families of solutions for   \ref{liouvhat} to the problem  of finding critical points of a functional $\Psi(\bbm[\xi])$ defined on a finite dimensional domain.

\
 
For $\bbm[\xi]=(\xi_1,\dots, \xi_N)$ let us introduce the functional
\beq\label{psi0}\begin{aligned}
\Psi(\bbm[\xi])&:= \sum_{j=1}^Nh(\xi_j, \xi_j)+\frac{1}{4\pi} \sum_{j=1}^N\log \tilde K(\xi_j)+\sum_{j,k=1\atop j\neq k}^NG(\xi_j,\xi_k) 
\end{aligned}
\eeq where $\tilde{K}$ is defined in \eqref{aaa},   $G$ denotes the Green function of $-\Delta_g$ over $\Sigma$ and $h$ its regular part as
in \eqref{0948}. $\Psi$ is well defined in the set
\beq\label{emme}{\cal M}^+:=(\Sigma^+\setminus\{p_1,\dots,p_\ell\})^N\setminus \Delta,\qquad \Delta:=\Big\{\bbm[\xi]\in \Sigma^N\,\Big|\,  \xi_j= \xi_k\,\hbox{ for some } j\neq k\Big\}.\eeq

\noindent The definition of $\Psi$ depends on the particular $\underline{\alpha}=(\alpha_1,\ldots,\alpha_m$) owing to \eqref{aaa}.
 To emphasize this fact sometimes we will write $\Psi_{\underline{\alpha}}(\bbm[\xi])$ in the place of  $\Psi(\bbm[\xi])$.

\

\noindent In order  to state the relation between the critical points of $\Psi$ and the solutions of \ref{liouvhat} let  us recall the notion of \textit{stable} critical point, which was introduced in \cite{li} in the analysis of concentration phenomena in nonlinear Schr\"odinger equations.

 \begin{definition}\label{defli} A critical point $\bbm[\xi]\in {\cal M}^+$ of $\Psi$ is \textit{stable} if for any neighborhood $U$ of $\bbm[\xi]$ in ${\cal M}^+$ there exists $\delta>0$ such that if $\|F-\Psi\|_{{\cal C}^1}\leq \delta$, then $F$ has at least one critical point in $U$. In particular, any (possibly degenerate) local minimum or maximum point is \textit{stable}, as well as  any non degenerate critical point and any  isolated critical point  with non-trivial local degree.
 \end{definition}

\

\noindent Next,  for $\bbm[\xi]=(\xi_1,\dots, \xi_N)\in{\cal M}^+$ we introduce the function
\beq\label{matA}\begin{aligned}
\mathcal A(\bbm[\xi])&:=4\pi\sum_{j=1}^N\tilde K(\xi_j)e^{8\pi h(\xi_j,\xi_j)+8\pi\sum_{k\neq j}G(\xi_j,\xi_k)}\bigg[\Delta_g \log\tilde K(\xi_j)+\frac{8\pi N}{|\Sigma|}-2\kappa_g(\xi_j)\bigg]\\
&=4\pi\sum_{j=1}^N \tilde K(\xi_j)e^{8\pi h(\xi_j,\xi_j)+8\pi\sum_{k\neq j}G(\xi_j,\xi_k)}\bigg[\Delta_g \log K(\xi_j)+\!\frac{8\pi N\!- \!4\pi\chi(\Sigma,\underline{\alpha})}{|\Sigma|}\!\bigg].\end{aligned}\eeq

\

Then, the variational reduction method developed in  \cite{espofigue} gives the following result, where the role of \emph{stable} critical points of $\Psi$ arises in the existence of solutions of \ref{liouvhat}. Even though in \cite{espofigue} only the case of positive orders is considered, one can easily check that  the proof continue to hold also for $\alpha_i>-1$.

\begin{proposition}[\cite{espofigue}]\label{prop:espofigue}
Let $N,m\in\mathbb N$. Assume that   $K:\Sigma\rightarrow \mathbb R$ is a  $C^2$ function  and  $p_i\in \Sigma$ for $i=1,\ldots, m$.  
Then for any  $-1<\alpha_{\star}<\alpha^\star$  there exists $\delta=\delta(\alpha^{\star},\alpha_{\star})>0$   such that if:
\begin{itemize}
\item[$a)$] 
 $\alpha_{\star}\leq\alpha_i\leq \alpha^{\star}$ for any $i=1,\ldots, m$,
\item[$b)$]  $\bbm[\xi^{\ast}]\in\mathcal M^+$ is such that $\mathcal A(\bbm[\xi^{\ast}])>0$ ($<0$ resp.),
\item[$c)$] $\bbm[\xi^{\ast}]$ is a stable critical point of $\Psi_{\underline{\alpha}}$, 
\end{itemize}
 then for all $\rho\in(8\pi N,8\pi N+\delta)$ ($\rho\in(8\pi N-\delta,8\pi N)$ resp.) there is a solution $v_\rho$ of \ref{liouvhat}.
 \\
 Moreover 
\beq\label{conv}\rho\frac{\tilde K(x)e^{v_\rho}}{\int_{\Sigma}\tilde K(x)e^{v_\rho}dV_g}\to 8\pi \sum_{j=1}^N\delta_{\xi_j^\ast}\quad\hbox{ as }\rho\to 8\pi N
\eeq 
in the measure sense.
\end{proposition}
We point out that the above proposition is stated here in a slightly more general way than in \cite[Theorem 1.1]{espofigue}; precisely in our formulation we   stress that the number $\delta$ can be chosen uniformly for bounded values of $\alpha_i$  away from $-1$.

\

\begin{remark}[Condition $b)$]
\label{rem:A<0} If the function $K$ and the orders $\alpha_1,\ldots,\alpha_m$ satisfy 
\begin{equation}\label{condizioneKperA<o}\displaystyle \sup_{\xi\in\Sigma^+}(\Delta_g\log K(\xi))<\frac{(4\pi\chi(\Sigma,\underline{\alpha})-8\pi N)}{|\Sigma|},
\end{equation}
then we have $\mathcal A(\bbm[\xi])< 0$ for all $\bbm[\xi]\in \cal M^+$, and, consequently, condition $b)$ in Proposition \ref{prop:espofigue} is satisfied. In particular \eqref{condizioneKperA<o} holds if there exists $\beta>0$ such that 
\beq\label{condiz}\sup_{\xi\in\Sigma^+}(\Delta_g\log K(\xi))\leq -\beta\quad\&\quad \chi(\Sigma,\underline{\alpha})> 2 N-\frac{\beta|\Sigma|}{4\pi}.\eeq
Observe that these two conditions are indeed assumed in Theorems \ref{thm:intro N+}, \ref{thm:intro noncontractible}, \ref{thm:intro contractible}.
In Section \ref{Section:ultima}  we will also prove other existence results for \ref{liouvhat} without assuming  \eqref{condiz} but
exhibiting classes of functions $K$, sign-changing or also positive, (and values of   $\alpha_i$'s) for
which $\cal A<0$ in suitable subsets of ${\cal M}^+$ where one can  find a local maximum of $\Psi_{\underline{\alpha}}$ (see Theorem  \ref{prop:example2}).
\end{remark}

\begin{remark}\label{rem:A>0} We notice that, since $K>0$ on $\Sigma^+$ and $K=0$ on $\partial \Sigma^+$, one cannot have $\Delta_gK>0$ on $\Sigma^+$, so it is not possible   to find a general reasonable explicit sufficient condition to guarantee ${\cal A}(\bbm[\xi])>0$ for all $\bbm[\xi]\in{\cal M}^+$ similar to \eqref{condiz} above.
Nevertheless  in Section \ref{Section:ultima}  we will provide  examples of functions $K$ for which  a stable critical point $\bbm[\xi]^*$ exists for $\Psi_{\underline{\alpha}}$ and satisfies  condition $\mathcal A(\bbm[\xi]^*)> 0$, 
yielding solvability of problem \ref{liouvhat} thanks to Proposition \ref{prop:espofigue} (see Theorem \ref{prop:example1}). \end{remark}

\noindent We discuss now sufficient conditions for assumption $c)$ in Proposition \ref{prop:espofigue} to hold. The following first result is immediate and  deals with the case when $\Sigma^+$ has a sufficiently large number of connected components. 
\begin{proposition}
\label{prop:ptocrit0} Let $N\in\N$ and $\alpha_1,\ldots, \alpha_m>-1$. Assume that $\Sigma^+$ consists of $N^+$ connected components with $N\leq N^+,$  hypotheses (H1), (H2) hold and, in addition, \beq\label{acca0}\alpha_i>0\quad \forall i=1,\ldots, \ell.\eeq Then
the functional $\Psi$ admits a local maximum $\bbm[\xi]^*=(\xi_1^*,\ldots, \xi_N^*)\in\mathcal M^+$
with each  point $\xi_j^*$ belonging to a separate  connected component of $\Sigma^+$.
\end{proposition}

\noindent The proofs of the next two results is quite involved and  will be developed in Section \ref{section:minmax}.

\begin{proposition}\label{prop:ptocrit1}
Let $N\in\N$ and $\alpha_1,\ldots, \alpha_m>-1$. Suppose that $\Sigma^+$ has a non contractible connected component,  hypotheses (H1), (H2), (H3), (H4) hold and, in addition,
\beq\label{acca5}\alpha_i\neq 0,1,\ldots, N-1\quad \forall i=1, \ldots,\ell
.\eeq Then $\Psi$ has a stable critical point $ \bbm[\xi]^*\in{\cal M}^+$.
\end{proposition}

\begin{proposition}\label{prop:ptocrit2}
Let $N\in\N$ and $\alpha_1,\ldots, \alpha_m>-1$. Suppose that  hypotheses (H1), (H2), (H3), (H4)  hold and, in addition, \beq\label{acca55}\alpha_i\neq 0,1,\ldots, N-1\quad \forall i=1, \ldots,\ell
,\eeq
\begin{equation}\label{ineq}
N\leq \ell +\sum_{i=1}^\ell [\alpha_i]^-.
\end{equation} 
Then $\Psi$ has a stable critical point  $ \bbm[\xi]^*\in {\cal M}^+$.
\end{proposition}

\

\noindent Propositions \ref{prop:ptocrit0}-\ref{prop:ptocrit1}-\ref{prop:ptocrit2} are at the core of this work, indeed by combining them with Proposition \ref{prop:espofigue} we get all our existence results and in particular Theorems \ref{thm:intro N+}, \ref{thm:intro noncontractible} and \ref{thm:intro contractible}, as we will see in the next section.

\

\

\section{Existence results for the general Liouville problem \ref{liouvhat}}\label{section:liouvhat}

\noindent In this section we provide the three main  existence results for the Liouville equation \ref{liouvhat}, from which we will deduce 
Theorems \ref{thm:intro N+}, \ref{thm:intro noncontractible} and \ref{thm:intro contractible} by choosing \[\rho=\rho_{geo}=4\pi\chi(\Sigma,\underline{\alpha}).\]
\noindent 

Let us begin with the first result which is a combination of Proposition \ref{prop:espofigue} and  Proposition \ref{prop:ptocrit0}.
\begin{theorem}\label{thm:0} Let $N\in\N$. Assume that $\Sigma^+$ consists of $N^+$ connected components with $N\leq N^+$ and   hypotheses (H1), (H2) hold. Then for any $-1<\alpha_{\star}<\alpha^{\star}$  there exists $\delta=\delta(\alpha^{\star},\alpha_{\star})>0$  such that, if $\alpha_1,\ldots, \alpha_m$ verify \eqref{acca0} and 
\begin{itemize}
\item[{\rm{(i)}}] $\alpha_{\star}\leq\alpha_i\leq\alpha^{\star}$ for any $i=1,\ldots,m$,
\item[{\rm{(ii)}}]  $\mathcal A>0$ ($<0$ resp.) in the set local maxima   of $\Psi$,
\end{itemize}
then for all $\rho\in(8\pi N,8\pi N+\delta)$ ($\rho\in(8\pi N-\delta,8\pi N)$ resp.) there is a solution $v_\rho$ of \ref{liouvhat}. 
Moreover there exists $ \bbm[\xi]^*\in {\cal M}^+$ such that \eqref{conv} holds.
\end{theorem}

\noindent Taking into account of  Remark \ref{rem:A<0}, Theorem \ref{thm:0} can be reformulated as follows.

\begin{corollary}\label{cor:0}
Let $N\in\N$. Assume that $\Sigma^+$ consists of $N^+$ connected components with $N\leq N^+$,   hypotheses (H1), (H2) hold and  $$\sup_{\xi\in\Sigma^+}(\Delta_g\log K(\xi))\leq -\beta$$ for some $\beta>0$.
Then for any $-1<\alpha_{\star}<\alpha^{\star}$  there exists $\delta=\delta(\alpha^{\star},\alpha_{\star})>0$ such that, if $\alpha_1, \ldots,\alpha_m$ verify \eqref{acca0}  and:
\begin{itemize}
\item[{\rm{(i)}}] $\alpha_{\star}\leq\alpha_i\leq\alpha^{\star}$ for any $i=1,\ldots,m$,
\item[{\rm{(ii)}}]  $\chi(\Sigma,\underline{\alpha})>2 N-\frac{\beta|\Sigma|}{4\pi},$
\end{itemize}
then for all $\rho\in(8\pi N-\delta,8\pi N)$ there is a solution $v_\rho$ of \ref{liouvhat}. 
Moreover there exists $ \bbm[\xi]^*\in {\cal M}^+$ such that \eqref{conv} holds.

\end{corollary}

\

\noindent Similarly combining Proposition \ref{prop:espofigue} and  Proposition \ref{prop:ptocrit1} we get the following.

\begin{theorem}\label{thm:1} Let $N\in\N$. Suppose that $\Sigma^+$ has a non contractible connected component and  hypotheses (H1), (H2), (H3), (H4) hold. Then for any $-1<\alpha_{\star}<\alpha^{\star}$  there exists $\delta=\delta(\alpha^{\star},\alpha_{\star})>0$ such that, if $\alpha_1,\ldots, \alpha_m$ verify \eqref{acca5} and 
\begin{itemize}
\item[{\rm{(i)}}] $\alpha_{\star}\leq\alpha_i\leq\alpha^{\star}$ for any $i=1,\ldots,m$,
\item[{\rm{(ii)}}]  $\mathcal A>0$ ($<0$ resp.) in the set of critical points  of $\Psi$,
\end{itemize}
then for all $\rho\in(8\pi N,8\pi N+\delta)$ ($\rho\in(8\pi N-\delta,8\pi N)$ resp.) there is a solution $v_\rho$ of \ref{liouvhat}. 
Moreover there exists $ \bbm[\xi]^*\in {\cal M}^+$ such that \eqref{conv} holds.

\end{theorem}

\noindent Proceeding similarly as above, using Remark \ref{rem:A<0} we also have:

\begin{corollary}\label{cor:1}  Let $N\in\N$. Suppose that $\Sigma^+$ has a non contractible connected component,  hypotheses (H1), (H2), (H3), (H4) hold and 
$$\sup_{\xi\in\Sigma^+}(\Delta_g\log K(\xi))\leq -\beta$$ for some $\beta>0$.
Then for any $-1<\alpha_{\star}<\alpha^{\star}$  there exists $\delta=\delta(\alpha^{\star},\alpha_{\star})>0$ such that, if $\alpha_1,\ldots, \alpha_m$ verify \eqref{acca5} and: 
\begin{itemize}
\item[{\rm{(i)}}] $\alpha_{\star}\leq\alpha_i\leq\alpha^{\star}$ for any $i=1,\ldots,m$,
\item[{\rm{(ii)}}]  $\chi(\Sigma,\underline{\alpha})>2 N-\frac{\beta|\Sigma|}{4\pi},$
\end{itemize}
then for all $\rho\in(8\pi N-\delta,8\pi N)$ there is a solution $v_\rho$ of \ref{liouvhat}.
Moreover there exists $ \bbm[\xi]^*\in {\cal M}^+$ such that \eqref{conv} holds.
\end{corollary}

\

\noindent Last combining Proposition \ref{prop:espofigue} and Proposition \ref{prop:ptocrit2} we obtain:

\begin{theorem}\label{thm:2}  Let $N\in\N$. Suppose that  hypotheses (H1), (H2), (H3), (H4)  hold.
Then for any $-1<\alpha_{\star}<\alpha^{\star}$  there exists $\delta=\delta(\alpha^{\star},\alpha_{\star})>0$ such that, if $\alpha_1,\ldots, \alpha_m$ verify \eqref{acca55}, \eqref{ineq} and 
\begin{itemize}
\item[{\rm{(i)}}] $\alpha_{\star}\leq\alpha_i\leq\alpha^{\star}$ for any $i=1,\ldots,m$,
\item[{\rm{(ii)}}]  $\mathcal A>0$ ($<0$ resp.) in the set of critical points  of $\Psi$,
\end{itemize}
then for all $\rho\in(8\pi N,8\pi N+\delta)$ ($\rho\in(8\pi N-\delta,8\pi N)$ resp.) there is a solution $v_\rho$ of \ref{liouvhat}. 
Moreover there exists $ \bbm[\xi]^*\in {\cal M}^+$ such that \eqref{conv} holds.
\end{theorem}

\noindent Once more, using Remark \ref{rem:A<0} we also have:

\begin{corollary}\label{cor:2}
Let $N\in\N$. Suppose that  hypotheses (H1), (H2), (H3), (H4)  hold 
 and $$\sup_{\xi\in\Sigma^+}(\Delta_g\log K(\xi))\leq -\beta$$ for some $\beta>0$.
 Then for any $-1<\alpha_{\star}<\alpha^{\star}$  there exists $\delta=\delta(\alpha^{\star},\alpha_{\star})>0$ such that, if $\alpha_1,\ldots, \alpha_m$ verify \eqref{acca55}, \eqref{ineq} and \begin{itemize}
\item[{\rm{(i)}}] $\alpha_{\star}\leq\alpha_i\leq\alpha^{\star}$ for any $i=1,\ldots,m$,
\item[{\rm{(ii)}}]  $\chi(\Sigma,\underline{\alpha})>2 N-\frac{\beta|\Sigma|}{4\pi},$
\end{itemize}
then for all $\rho\in(8\pi N,8\pi N+\delta)$ ($\rho\in(8\pi N-\delta,8\pi N)$ resp.) there is a solution $v_\rho$ of \ref{liouvhat}. 
Moreover there exists $ \bbm[\xi]^*\in {\cal M}^+$ such that \eqref{conv} holds.

\end{corollary}

\

\noindent Observe that 
Theorem \ref{thm:intro N+}, Theorem \ref{thm:intro noncontractible} and Theorem \ref{thm:intro contractible} follow immediately from Corollary \ref{cor:0}, Corollary \ref{cor:1} Corollary \ref{cor:2}, respectively,  by taking $\alpha^\star=2 N-\chi(\Sigma)+m$ and $\rho=\rho_{geo}$. 

\
Thus in order to achieve the existence results for problem \ref{liouvgeom} such as the ones predicted by Theorem \ref{thm:intro N+}, Theorem \ref{thm:intro noncontractible} and Theorem \ref{thm:intro contractible}, it remains to prove  Propositions \ref{prop:ptocrit1}--\ref{prop:ptocrit2} (the proof of Proposition  \ref{prop:ptocrit0} is immediate). This will be accomplished in the next section.

\

\section{The min-max scheme}\label{section:minmax}
The discussion in the previous section  implies that our problem reduces now to investigate the existence of stable critical points for the reduced energy $\Psi$ in order to prove   Proposition \ref{prop:ptocrit1} and \ref{prop:ptocrit2}.  In this section we will apply a max-min argument to characterize a
topologically nontrivial  critical value of this function $\Psi$ in the set ${\cal M}^+$. Since $\tilde K$ is defined by \eqref{aaa},  $\Psi$ actually becomes 
$$\Psi( \bbm[\xi])={\cal H}(\bbm[\xi])+\frac{1}{4\pi} \sum_{j=1}^N\log  K(\xi_j)-
 \sum_{i=1}^\ell \alpha_i \sum_{j=1}^N G(\xi_j,p_i)
+\sum_{j,k=1\atop j\neq k}^NG(\xi_j,\xi_k).$$
where ${\cal H}$ is a smooth term on $(\Sigma^+)^N$, precisely $${\cal H}(\bbm[\xi]):=\sum_{j=1}^Nh(\xi_j, \xi_j)-
 \sum_{i=\ell+1}^m \alpha_i \sum_{j=1}^N G(\xi_j,p_i)\in {C}^1((\Sigma^+)^N).$$

Let us briefly outline the variational argument we are going to set up, which consists in two parts. 

First we will construct  sets ${\cal B}, \,{\cal B}_0, {\cal D}\subset {\cal M}^+$ satisfying the following two properties:
\begin{enumerate}
\item [(P1)] $\mathcal D$ is open,   ${\cal B}$ and ${\cal B}_0$ are
compact,  ${\cal B}$ is connected and
\begin{equation*}{\cal B}_0\subset {\cal B}\subset \mathcal D\subset {\overline {\mathcal D}}\subset {\cal M}^+;\end{equation*}
\item[(P2)] let us set ${\mathcal F}$ to be the class of all continuos maps $\gamma:{\cal B}\to {\cal D}$ with the property that there exists a continuos homotopy $\Gamma:[0,1]\times {\cal B}\to {\cal D}$ such that:       $$\Gamma(0,\cdot)=id_{\cal B},\quad \Gamma(1,\cdot)=\gamma,\quad \Gamma(t,\bbm[\xi])=\bbm[\xi]\;\;\forall t\in [0,1],\,\forall \bbm[\xi]\in {\cal B}_0;$$  then
\begin{equation}\label{mima}{ \Psi}^*:=\sup_{\gamma\in{\mathcal F}}\min_{\hbox{\scriptsize$\bbm[\xi]$}\in
{\cal B}}\Psi(\gamma(\bbm[\xi]))<\min_{\hbox{\scriptsize$\bbm[\xi]$}\in {\cal B}_0} \Psi(\bbm[\xi]);\end{equation}
\end{enumerate}
Secondly, we need to exclude the possibility that the critical point is placed on the boundary of our domain, and precisely we need that:

\begin{enumerate}
 \item[(P3)] for every $\bbm[\xi]\in\partial\mathcal D$ such that $\Psi(\bbm[\xi])={\Psi}^*$, $\partial \mathcal D$ is smooth at $\bbm[\xi]$ and
there exists a vector $\tau_{\hbox{\scriptsize$\bbm[\xi]$}}$ tangent to $\partial\mathcal
D$ at $\bbm[\xi]$ so that $\tau_{\hbox{\scriptsize$\bbm[\xi]$}}\cdot\nabla \Psi(\bbm[\xi])\neq 0$.

\end{enumerate}

\bigskip

Under these assumptions a critical point $\bbm[\xi]\in
{\mathcal D}$ of $\Psi$ with $\Psi(\bbm[\xi])=\Psi^*$
exists, as a standard deformation argument involving the gradient
flow of $\Psi$ shows.  
 Moreover, since properties (P2)-(P3) continue to hold also for a functional which is ${C}^1$-close to $\Psi$, then  such critical point will  \textit{survive} small ${\cal C}^1$-perturbations and, consequently, will be \textit{stable} in the sense of Definition \ref{defli}.
 
\

Hence, once properties (P1)-(P2)-(P3) are established, for suitable sets ${\cal B}, \,{\cal B}_0$ and ${\cal D}$, Propositions \ref{prop:ptocrit1} and \ref{prop:ptocrit2} would follow. We will prove (P1)-(P2)-(P3) in Sections \ref{section:P1}, \ref{section:P2} and \ref{section:P3} respectively.

\subsection{Definition of ${\cal B}$, ${\cal B}_0$, and proof of (P1)}\label{section:P1}
 To establish property (P1), we define 
  \beq\label{didi}{\mathcal D}=\big\{\bbm[\xi]\in{\cal M}^+\;
\big|\,\Phi(\bbm[\xi])  >-M \big\}\eeq 
where $M>0$ is a sufficiently large number yet to be chosen and 
$$\begin{aligned}\Phi(\bbm[\xi]):&=
{\cal H}(\bbm[\xi])+\frac{1}{4\pi}\sum_{j=1}^N \log K (\xi_j)-\sum_{i=1}^{\ell}\alpha_i\sum_{j=1}^N G(\xi_j,p_i)
-\sum_{j,k=1\atop j\neq k}^NG(\xi_j, \xi_k).\end{aligned}$$
 By using the properties of the functions $K,\, G$ it is easy to check  that $\Phi$ satisfies \begin{equation}\label{coercivi}\Phi(\bbm[\xi])\to -\infty\hbox{ as }\bbm[\xi]\to \partial{\cal M}^+,\end{equation} and this implies that  ${\mathcal
D}$ is compactly contained in ${\cal M}^+$. 
In order to define ${\cal B}$, we fix 
\beq\label{pap}\sigma_1,\dots,\sigma_N\subset\Sigma^+ \setminus \{p_1,\dots,p_\ell\}\eeq $N$  (not necessarily distinct) simple, closed curves in $\Sigma^+$ which do not intersect any of the singular sources $p_i$. Next we fix   $$\bbm[\xi]_0=(\xi_1^0, \ldots, \xi_N^0)\in \sigma_1\times\ldots\times\sigma_N,\qquad \xi_j^0\neq \xi_k^0\;\;\;\forall j\neq k $$ a $N$-tuple of 
$N$ distinct points. The exact choice of curves $\sigma_j$ and points $\xi_j^0$ will be specified later and will depend on the topology of $\Sigma^+$.
We introduce  the  set 
\beq\label{openset}\left\{\bbm[\xi]\in \Sigma^+ \,\Big|\,\xi_j\in \sigma_j \;\;\&\;\;d_g(\xi_j,\xi_k)>M^{-1}   \quad \forall j\neq k\right\}.\eeq
In principle, we do not know
whether \eqref{openset} is connected or not, so we will choose 
 a convenient connected component $W$. Since $\xi_j^0\neq \xi_k^0$ for $j\neq k$, then $\bbm[\xi]_0$ belongs to \eqref{openset} provided that $M$ is sufficiently large. 
 Now we are in conditions of defining ${\cal B}$ and ${\cal B}_0$: 
$$W:=\hbox{ the connected component of \eqref{openset} containing } \bbm[\xi]_0,$$
$${\cal B}:=\overline W,\quad {\cal B}_0=\Big\{\bbm[\xi]\in {\cal B}\,\Big|\, \min_{j\neq k}d_g(\xi_j,\xi_k)=M^{-1}\Big\}.$$
${\cal B}$ is clearly connected and ${\cal B}_0\subset {\cal B}$. Moreover by construction   we get that the $N$-tuple of  points in \eqref{openset} are uniformly distant from the sources $p_i$ and as well as from the boundary $\partial \Sigma^+$ thanks to \eqref{pap}, therefore  \beq\label{iuni}\sum_{j=1}^N \log K(\xi_j)
=O(1),\quad \sum_{i=1}^{\ell}\alpha_i\sum_{j=1}^N G(\xi_j,p_i)=O(1)\;\hbox{  in }{\cal B}\eeq with the above quantity $O(1)$ uniformly bounded independently of $M$. On the other hand in the set ${\cal B}$ we also have $G(\xi_j,\xi_k)\leq \log M+C$ for $j\neq k$ by \eqref{0948}. Consequently for large $M$ we also have ${\cal B}\subset {\cal D}$. We have thus proved property (P1). 

\subsection{Proof of (P2).}\label{section:P2} During this section we assume that assumptions  (H1), (H2), (H3), (H4) hold. 
 We begin by providing the following crucial intersection property which is an easy consequence of a topological degree argument. 
\begin{lemma}\label{inters} For any $j=1,\dots,N$ let ${\cal P}_j $  be a retraction of $\Sigma^+ \setminus \{p_1,\dots,p_\ell\}$ onto $\sigma_j$, i.e. $\mathcal{P}_j: \Sigma^+ \setminus \{p_1,\dots,p_\ell\} \to \sigma_j$ is a continuous map so that $\mathcal{P}_j  \big|_{\sigma_j}=\hbox{id}_{\sigma_j}$. Then for any  $\gamma \in \mathcal{F}$ there exists $\bbm[\xi]_\gamma^* \in {\cal B}$ such that $$\mathcal{P}_j(\gamma_j(\bbm[\xi]_\gamma^*))=\xi_j^0\quad \forall j=1,\dots,N.$$
\end{lemma}
\begin{proof}
Let $\gamma\in {\mathcal F},$ namely $\gamma:{\cal B}\to {\cal D}$ is a continuous map such that  there exists a continuos homotopy $\Gamma:[0,1]\times {\cal B}\to {\cal D}$ satisfying:       $$\Gamma(0,\cdot)=id_{\cal B},\quad \Gamma(1,\cdot)=\gamma,\quad \Gamma(t,\bbm[\xi])=\bbm[\xi]\;\;\forall t\in [0,1],\,\forall \bbm[\xi]\in {\cal B}_0.$$ 
Extend $\Gamma$ continuously from ${\cal B}$  to $\bbm[\sigma]:=\sigma_1\times\ldots\times\sigma_N$ 
as $\tilde \Gamma:[0,1]\times\bbm[\sigma]\to {\cal D}$ defined simply as 
$$\widetilde \Gamma(t,\bbm[\xi])=\Gamma(t,\bbm[\xi])\;\;\hbox{ if }\bbm[\xi]\in {\cal B},\;\;\;\; \widetilde \Gamma(t,\bbm[\xi])=\bbm[\xi]\;\;\hbox{ if }\bbm[\xi]\in \bbm[\sigma]\setminus {\cal B}.$$
Notice that ${\cal B}_0$ is the topological boundary of ${\cal B}$ relative to $\bbm[\sigma]$, then $\tilde\Gamma$ is a continuos map and$$\tilde{\Gamma}(0,\cdot)=id_{\bbm[\sigma]},\quad \tilde{\Gamma}(t,\cdot)\big|_{{\cal B}_0}=id_{{\cal B}_0} ,\quad \tilde{\Gamma}(t,\cdot)\big|_{\bbm[\sigma]\setminus {\cal B}}=id_{\bbm[\sigma]\setminus {\cal B}}\quad \forall t\in [0,1].
$$
Set $\gamma=(\gamma_1,\ldots, \gamma_N) $ and  $\tilde\Gamma=(\tilde\Gamma_1,\ldots, \tilde\Gamma_N)$ with $\gamma_j:{\cal B}\to \Sigma^+$ and $\tilde\Gamma_j: [0,1]\times \bbm[\sigma]\to \Sigma^+$. Then the map ${\cal S}:[0,1]\times  \bbm[\sigma]\to\bbm[\sigma]$ with components
$${\cal S}_j(t,\bbm[\xi])=({\cal P}_j\circ \tilde\Gamma_j)(t,\bbm[\xi]) ,\quad j=1,\ldots, N,$$ is continuous and 
satisfies \beq\label{zuu}{\cal S}(0,\cdot)=id_{\bbm[\sigma]},\quad {\cal S}(t,\cdot)\big|_{\bbm[\sigma]\setminus {\cal B}}=id_{\bbm[\sigma]\setminus {\cal B}}\quad \forall t\in [0,1].
\eeq
In order to apply a degree argument, we can identify each $\sigma_j$, $j=1,\dots,N$, with $\mathbb{S}^1$ through a suitable homeomorphism, and then regard ${\cal S}$ as a map  $[0,1] \times (\mathbb{S}^1)^N \to (\mathbb{S}^1)^N$ with ${\cal S}(0,\cdot)=id_{(\mathbb{S}^1)^N}$. We consider the annulus in $\R^2$ 
$$U:=\left\{u\in \R^2\,\Big|\, \frac12<|u|<2\right\}.$$ Then we extend ${\cal S}$ from $ (\mathbb{S}^1)^N$ to $\overline U^N$ as $\tilde{\cal S}$ having components 
$$\tilde{\cal S}_j(t,\bbm[u])=|u_j|{\cal S}_j\Big(t, \frac{u_1}{|u_1|},\ldots,\frac{u_N}{|u_N|}\Big), \quad \forall \bbm[u]=(u_1,\ldots, u_N)\in\overline U^N. $$ 
Notice that $\frac{u_j}{|u_j|}\in \mathbb{S}^1$ for $u_j\in \overline U$, so $\tilde{\cal S}_j$ is well defined.
Clearly $\tilde{\cal S}$ is a continuous map by construction and 
$$\tilde{\cal S}(0,\cdot)=id_{\overline U^N}.$$
Moreover the definition of $\tilde{\cal S}$ yields  \beq\label{fre}|\tilde{\cal S}_j(t,\bbm[u])|=|u_j|\quad \forall t\in [0,1], \forall \bbm[u]\in\overline {U}^N \eeq and, consequently,
$$\tilde{\cal S}\big(t,U^N\big)\subset U^N\quad \forall t\in [0,1]$$
and $$\tilde{\cal S}\big(t, \partial \big(U^N\big)\big)\subset \partial \big(U^N\big)\quad \forall t\in [0,1].$$
Once we have proved the crucial property that $\tilde{\cal S}$  maps the boundary $\partial\big (U^N\big)$ into itself, now we are in the position to apply  a topological degree argument: indeed,  the homotopy invariance  gives that if $\bbm[u]\in U^N$ then $\deg(\tilde{\cal S}(1, \cdot), U^N, \bbm[u])=\deg(\tilde{\cal S}(0, \cdot), U^N, \bbm[u])=\deg (id, U^N, \bbm[u])=1$. 
In particular$$\deg(\tilde{\cal S}(1, \cdot), U^N, \bbm[\xi]^0)=1$$ where $\bbm[\xi]^0 \in (\mathbb{S}^1)^N$ corresponds to the original $\bbm[\xi]_0 \in \bbm[\sigma]$ through the identifications of each $\sigma_j$ with $\mathbb{S}^1$. Then, there exists $\bbm[u]^*=(u_1^*,\ldots, u_N^*)\in U^N$ so that
$$\tilde{\cal S}(1, \bbm[u]^*)=\bbm[\xi]^0.$$
Thanks to \eqref{fre} we get $\bbm[u]^*\in (\mathbb{S}^1)^N$, which, in  turn, implies $${\cal S}(1, \bbm[u]^*)=\tilde{\cal S}(1, \bbm[u]^*)=\bbm[\xi]^0.$$
Getting back to $\bbm[\sigma]$ again by the isomorphism $\bbm[\sigma]\approx  (\mathbb{S}^1)^N$, we deduce the existence of   $\bbm[\xi]^*=(\xi_1^*,\ldots, \xi_N^*)\in \bbm[\sigma]$ such that $${\cal S}(1, \bbm[\xi]^*)=\bbm[\xi]_0.$$ We claim that $\bbm[\xi]^*\in {\cal B}$:  otherwise, if $\bbm[\xi]^* \in \bbm[\sigma] \setminus {\cal B}$, then ${\cal S}(1,\bbm[\xi]^*)=\bbm[\xi]^*$ by \eqref{zuu}, which would lead to $\bbm[\xi]^*=\bbm[\xi]_0 $, and this provides a contradiction with $\bbm[\xi]_0\in {\cal B}$. So, $\bbm[\xi]^*\in {\cal B}$ and $$\mathcal{P}_j(\gamma_j(\bbm[\xi]^*))={\cal P}_j(\Gamma_j(1, \bbm[\xi]^*))={\cal S}_j(1,\bbm[\xi]^*)=\xi_j^0\quad\forall j=1,\dots,N.$$\end{proof}

Now  we are going to prove (P2). 
 The definition of the max-min value $\Psi^*$ in \eqref{mima} depends on the particular $M>0$ chosen in \eqref{didi}. To emphasize this fact we denote this max-min value by $\Psi^*_M$. In the remaining part of this section we will prove that (P2) holds for $M$ sufficiently large. To this aim we need the estimate for $\Psi^*_M$ provided by the following two propositions 
 which prove the uniform boundedness (with respect to $M$)  under the assumptions of Proposition \ref{prop:ptocrit1} and Proposition \ref{prop:ptocrit2}, respectively. 
 
For the sake of simplicity, in the proofs we  will use the additional assumption that $$\Sigma^+ \hbox{ is connected}.$$ This assumption is made without loss of generality in this framework: indeed, if $ \Sigma^+$ is not connected, then  it is sufficient  to replace $\Sigma^+$ by one of its  connected components in the definition of the set ${\cal M}^+$ in \eqref{emme}, and then confining the
search of a critical point for $\Psi$ to such a component. 

\begin{remark}\label{2.5} Anyway let us stress that if $\Sigma^+$ is not connected, then the results of Proposition \ref{prop:ptocrit1} and Proposition \ref{prop:ptocrit2} may possibly be improved (allowing larger values of $N$ in Proposition \ref{prop:ptocrit2}, for instance) by suitably gluing the construction in each  connected components. Anyway the optimal results for non-connected surface would require  some more technicality  and we will not
comment on this issue any further.
\end{remark} 
\begin{proposition}\label{loweresti}
Assume that $\Sigma^+$ is connected and   non contractible. Then the quantity $\Psi^*_M$ is bounded independently  of the large number $M$ used to define ${\cal D}$, namely there exist two constants $c, C$ independent of $M$ such that 
\begin{equation}\label{rox}c\leq\Psi^*_M
\leq C.\end{equation}
\end{proposition}
\begin{proof}
To prove the lower boundedness it is sufficient to take $\gamma=id_{\cal B}$ in the definition \eqref{mima}:
\beq\label{lower}\Psi^*_M\geq \min_{\hbox{\scriptsize$\bbm[\xi]$}\in {\cal B}}\Psi(\bbm[\xi])\geq \min_{\hbox{\scriptsize$\bbm[\xi]$}\in {\cal B}}\bigg(\frac{1}{4\pi}\sum_{j=1}^N \log K(\xi_j)-\sum_{i=1}^{\ell}\alpha_i\sum_{j=1}^N G(\xi_j,p_i)-C\bigg) .\eeq	 As we have already observed in \eqref{iuni},   the function in the bracket is uniformly bounded 
in
the set ${\cal B}$
independently of $M$.

 To get an upper estimate for the max-min value we need that a crucial intersection property is accomplished, and this will follow from Lemma \ref{inters} for a  suitable choice of curves $\sigma_j$ and points $\xi_j^0$.  Since such a choice depends on the topological properties of $\Sigma$,  in order to perform the geometrical construction  it is convenient to distinguish the two cases $${\cal G}(\Sigma)=0\hbox{ and }{\cal G}(\Sigma)>0, $$ 
where  ${\cal G}(\Sigma)$ denotes the genus of $\Sigma$. 

 Before going on we observe that the thesis of the proposition is invariant under diffeomorphism: more precisely, assume that    $\omega:\overline{\Sigma^+}\to \omega(\overline{\Sigma^+})$ is a diffemomorphism   and suppose that we have proved the thesis   for the functional  $\Psi_M \circ\omega^{-1}(\xi')$ defined for $\xi'\in \omega(\overline{\Sigma^+})$ with the corresponding sets $\omega({\cal D}), \omega({\cal B}),\omega( {\cal B}_0)$; then, denoting by $g'$ the metric on $\omega(\overline{\Sigma^+})$ and setting $\xi':=\omega(\xi)$, we have    \beq\label{deus}c d_{g'}(\xi_j',\xi_k')\leq d_g(\xi_j,\xi_k)\leq Cd_{g'}(\xi_j',\xi_k').\eeq  
 Since the unbounded terms in the definition of $\Psi_M$ just involve the logarithm of the distance function or the logarithm of $K$,  then thanks to \eqref{deus}  the thesis continues to hold for our original functional $\Psi_M$. 
 
So, in the remaining part of the proof without loss of generality we may  replace $\Sigma^+$ with a topologically equivalent surface.

\begin{enumerate}
\item[]{\bf{Case I}: ${\cal G}(\Sigma)=0$}.
\end{enumerate}

The case of genus zero corresponds to a surface $\Sigma$ which is a topological sphere.  Then $\Sigma^+$ turns out to be  diffeomorphic to a planar domain which is non contractible.  So let us assume  that $\Sigma^+$ coincides with a  planar domain with a spherical hole of radius $1$: more precisely$$\Sigma ^+\equiv\Omega\subset \R^2,$$ and $$B_1:=\{x^2+y^2<1
\}\hbox{ is a connected component of }\R^2\setminus \overline\Omega.$$

In this case the  construction we are going to set up is based on a similar  argument carried out in \cite{DKM} in a Euclidean context.

Let us fix a radius $\rho>1$ sufficiently close to $1$ in such a way that the circle centered at $0$ with radius $\rho$ is contained in $\Omega$: $$\sigma:=\{x^2+y^2=\rho^2\}\subset \Omega,$$ and $\sigma$ does not intersect any of the singular points $p_i$:
$$p_i\not\in\sigma\quad\forall i=1,\ldots,\ell.$$ We construct a retraction of  $\Sigma^+\equiv\Omega$ onto $\sigma$ by simply projecting along rays starting from $0$:
$${\cal P}(\xi)=\rho\frac{\xi}{|\xi|}\quad \forall \xi\in \Omega.$$ 
Then we apply Lemma \ref{inters} by taking $$\sigma_j=\sigma,\qquad {\cal P}_j={\cal P}\quad \forall j=1,\ldots, N,$$
and we find that for any $\gamma\in {\cal F}$ there exists $\bbm[\xi]^*_\gamma\in {\cal B}$ such that 
$$\gamma_j(\bbm[\xi]^*_\gamma)\in {\cal P}^{-1}(\xi^0_j)\quad\forall j=1,\ldots, N.$$
By construction  the fibers of ${\cal P}$ are half-lines emanating from zero and  are well-separated thanks to the presence of the hole $B_1$, then, since $\xi_j^0\neq \xi_k^0$ for $j\neq k$, there exists $\mu>0$ such that 
\beq\label{fiber}{\rm{dist}}_{\rm{eucl}}({\cal P}^{-1}(\xi^0_j),{\cal P}^{-1}(\xi^0_k))\geq \mu \quad \forall j\neq k\eeq
which implies $$G(\gamma_j(\bbm[\xi]^*_\gamma), \gamma_k(\bbm[\xi]^*_\gamma))=O(1)\quad \forall j\neq k$$with the above quantity $O(1)$ uniformly bounded independently of $\gamma$. So an upper bound on $\Psi_M^*$ is obtained by evaluating on $\gamma(\bbm[\xi]^*_\gamma)$ as follows:
$$\begin{aligned}\min_{\hbox{\scriptsize$\bbm[\xi]$}\in {\cal B}}\Psi(\gamma(\bbm[\xi]))&
\leq \Psi( \gamma(\bbm[\xi]^*_\gamma))\leq 
\sum_{j,k=1\atop j\neq k}^NG(\gamma_j(\bbm[\xi]^*_\gamma), \gamma_k(\bbm[\xi]^*_\gamma))+C \leq C.
\end{aligned}$$
 Hence,  by taking the supremum for all the maps $\gamma\in{\cal F}$,   we conclude that the max-min value $\Psi^*_M$ is bounded above  independently of $M$, as desired. 

\medskip

\begin{enumerate}
\item[]{\bf{Case II}: ${\cal G}(\Sigma)>0$}.
\end{enumerate}
According to the classification of compact connected orientable surfaces (see \cite[Theorem 3.7, page 217]{hir}) we have that $\overline{\Sigma^+}$ is diffeomorphic to the surface obtained from an orientable closed surface  by removal of the interiors of $k$ disjoints disks. So let us assume that $$\Sigma\setminus \overline{\Sigma^+}\hbox{ is the disjoint union of  } k \hbox{ open disks}.$$
Moreover, since the genus of $\Sigma$ is positive, up to a new diffeomorphism we can also assume that $\Sigma$ is embedded in $\R^3$ and  satisfies 
\beq\label{toro}
\sigma:=\{x^2+y^2=1,\;z=0\} \subset \Sigma,\:\: \Sigma\cap \{x^2+y^2<1\}=\emptyset,\eeq \beq\label{toro2}\sigma\subset \Sigma^+,\;\;\sigma \cap \{p_1,\dots,p_\ell\}=\emptyset. \eeq
This is quite obvious when ${\cal G}(\Sigma)=1$: indeed, in this case $\Sigma$ is diffeomorphic to the torus
\beq\label{toro3}\{(x,y,z)\in \R^3\,|\, (\sqrt{x^2+y^2}-2)^2+z^2=1\}\eeq
which satisfies  \eqref{toro}; moreover,
possibly slightly perturbing the diffeomorphism, we can always assume that the singular sources $p_i$ ($i=1,\dots,\ell$) in $\Sigma^+$ do not belong to $\sigma$ and that $\sigma$ does not intersect
the $k$ disks of $\Sigma\setminus \overline{\Sigma^+}$, so that \eqref{toro2} holds. When ${\cal G}(\Sigma)=m \geq 2$, $\Sigma$ is diffeomorphic to  the connected sum  of $m$ torii, obtained by gluing in a smooth way the torus \eqref{toro3} with other $m-1$ torii outside the cylinder $\{x^2+y^2\leq 2\}$. Also in this case, $\Sigma$ satisfies properties \eqref{toro}-\eqref{toro2}. 

In what follows we adapt some argument used in \cite{tea&pp} for $K$ positive. Notice that the above assumptions \eqref{toro}-\eqref{toro2} are crucial to define a retraction of $\Sigma^+$ onto $\sigma$ as
\begin{eqnarray} 
\mathcal P (x,y,z)=\left( \frac{x}{x^2+y^2},\frac{y}{x^2+y^2},0\right). \label{mapP}
\end{eqnarray}
Indeed, by \eqref{toro}-\eqref{toro2} the map $\mathcal P:\: \Sigma^+ \to \sigma$ is well-defined and continuous with $\mathcal P\big|_\sigma=id_\sigma$.
Then we apply Lemma \ref{inters} by taking $$\sigma_j=\sigma,\qquad {\cal P}_j={\cal P}\quad \forall j=1,\ldots, N,$$
and we find that for any $\gamma\in {\cal F}$ there exists $\bbm[\xi]^*_\gamma\in {\cal B}$ such that 
$$\gamma_j(\bbm[\xi]^*_\gamma)\in {\cal P}^{-1}(\xi^0_j)\quad\forall j=1,\ldots, N.$$
Let us investigate the structure of the fibers of ${\cal P}$ in this case:  the fibers of ${\cal P}$ lie on  vertical half-planes starting from the $z$-axis  and their (euclidean) distance from the $z$-axis is greater than $1$ in view of \eqref{toro}. Then they are  well-separated, so  \eqref{fiber} is satisfied and we conclude as in the previous case.

\end{proof}

\begin{proposition}\label{loweresti2}
Assume that $\Sigma^+$ is connected and contractible and that the inequality \eqref{ineq} is satisfied. Then the same thesis of Proposition \ref{loweresti} holds. \end{proposition}
\begin{proof}

$\Sigma^+$ turns out to be diffeomorphic to a two dimensional domain. So, since the thesis of the proposition is invariant under diffeomorphism as we have observed at the beginning of the proof in Proposition \ref{loweresti}, from now on let us assume $$\Sigma^+\equiv \Omega\subset \R^2. $$ We are in the position to adapt the arguments in \cite{dap} for the following geometrical construction. 

A lower bound on $\Psi^*_M$ follows by taking $\gamma=id_{\cal B}$ and reasoning exactly as in \eqref{lower}. 

Let us focus on finding an upper estimate for $\Psi_M^*$. 
Hereafter we will  often use the complex numbers to identify the points in $\R^2$ and we will denote by ${\rm i}$ the imaginary unit.
First of all let us  fix  angles $\theta_i$ ($i=1,\ldots,\ell$)  and  a number $\delta\in (0,\frac{\pi}{2})$ sufficiently small such that 
the cones  \beq\label{di1}\big\{p_i+\rho e^{{\rm i}(\theta_i+\theta)}\,\big|\, \rho\geq0, \,\theta \in [-\delta, \delta]\big\},\,\quad i=1,\ldots,\ell\eeq are disjoint from one another.  We point out that such choice of angles always exists since the set of singular sources  $p_1,\ldots,p_\ell$ is finite. 
Possibly decreasing $\delta$, we may also assume \beq\label{di2}S_\delta(p_i)\subset \Omega,\qquad |p_i-p_r|>2\delta\;\; \quad\forall i,r=1,\ldots, \ell,\,i\neq r.\eeq where $S_\delta(p_i)$ denotes the circle in $\R^2$  with center $p_i$ and radius $\delta$. 
According to assumption \eqref{ineq} we may split  $N=N_1+N_2+\ldots+ N_\ell$ with $N_i\in\N$ satisfying \beq\label{crik}0\leq N_i\leq 1+[\alpha_i]^{-}\quad \forall i.\eeq  Next we split $\{1,\ldots, N\}=I_1\cup \ldots\cup I_\ell$ where 
$$\begin{aligned}&I_1=\{1,2,\ldots, N_1\}, \\ &I_2=\{N_1+1,N_1+2,\ldots, N_1+N_2\},\\ &\ldots\\ &
I_i=\{N_1+\ldots +N_{i-1}+1,\ldots, N_1+\ldots+N_i\},\\ &\ldots
\\ &I_\ell=\{N_1+\ldots +N_{\ell-1}+1,\ldots, N\}.\end{aligned}$$
Then we  set $$\sigma_j:=S_\delta(p_i),\qquad  {\cal P}_j(\xi)=p_i+\delta\frac{\xi-p_i}{|\xi-p_i|} \;\;\;\forall j\in I_i,\; i=1,\ldots, \ell.$$
Now we fix $N$-tuple
$$\bbm[\xi]_0=(\xi_1^0,\ldots, \xi_N^0)$$  by
\beq\label{defxi0}\xi_j^0=p_i+\delta e^{{\rm i}(\theta_{i}+j\frac{\delta}{N})}\in\sigma_j\;\; \;\;\;\forall j\in I_i,\; i=1,\ldots, \ell.\eeq
Clearly ${\cal P}_j:\Omega\setminus \{p_i\}\to \sigma_j$ defines a retraction onto $\sigma_j$. Then  Lemma \ref{inters} applies with this choice of $\xi_j^0, \sigma_j,{\cal P}_j$ and gives that for any $\gamma\in {\cal F}$ there exists $\bbm[\xi]^*_\gamma\in {\cal B}$ such that, setting $z_j= \gamma_j(\bbm[\xi]^*_\gamma)$, 
$${\cal P}_j(z_j)= \xi^0_j\quad\forall j=1,\ldots, N,$$
which implies $$
\frac{z_j-p_i}{|z_j-p_i|}=e^{{\rm i}(\theta_{i}+j\frac{\delta}{N})}\;\;\;\forall j\in I_i,\;i=1,\ldots, \ell.$$

Let us observe that in this case the fibers of $\mathcal P_j$ are half-lines emanating from $p_i$ and the assumption \eqref{ineq}
is required to get a control on the energy when two or more  components of $\bbm[z]=(z_1,\ldots, z_N)$  collapse onto $p_i$, which represents a crucial point to  establish	the uniform   	boundedness	from above of $\Psi^*_M$. 
Indeed, by construction we obtain
$$|z_j-z_k|\geq |z_j-p_i|\sin \frac{\delta}{2N}\quad \forall j,k\in I_i,\;j\neq k\;\;\;(i=1,\ldots, \ell).$$
Moreover, for any $j\in I_i$  we have that $\xi_j$ belongs to the cone \eqref{di1}. This  implies that $$| z_j-z_k|\geq \mu\quad \forall j\in I_i,\, k\in I_r,\;i\neq r$$ where the value  $\mu$ depends only  on the choice of the angles $\theta_i$ and the number $\delta$.  Combining these facts with \eqref{0948} 
  we may  estimate
\beq\label{stimm}\begin{aligned}\Psi(\bbm[z])&
\leq- \sum_{i=1}^\ell\frac{\alpha_{i}}{2\pi} \sum_{j=1}^N\log\frac{1}{|z_j-p_i|}+\frac{1}{2\pi}\sum_{j,k=1\atop j\neq k}^N\log\frac{1}{|z_j-z_k|}+C\\ &\leq -
\frac{1}{2\pi}\sum_{i=1}^\ell\alpha_{i} \sum_{j=1}^N\log\frac{1}{|z_j-p_i|}+\frac{1}{2\pi}\sum_{i=1}^\ell\sum_{j,k\in I_i\atop j\neq k}\log\frac{1}{|z_j-z_k|} +C. 
\end{aligned}
\eeq
For a fixed $i\in\{1,\ldots, \ell\}$ and $j\in I_i$ we have
$$\begin{aligned}&-\alpha_{i}\log\frac{1}{|z_j-p_i|}+\sum_{k\in I_i\atop k\neq j}\log\frac{1}{|z_j-z_k|}\\ &\leq -\alpha_{i}\log\frac{1}{|z_j-p_i|}+(N_i-1)\log\frac{1}{|z_j-p_i|}-(N_i-1)\log \sin\frac{\delta}{2N}. 
\end{aligned}$$
Since $\alpha_{i}> N_i-1$ by \eqref{crik}, the above quantity is uniformly bounded above. 
Combining this with \eqref{stimm} we deduce a uniform upper bound for $\Psi$ on the range of $\gamma$:
$$\min_{\hbox{\scriptsize$\bbm[\xi]$}\in
{\cal B}}\Psi(\gamma(\bbm[\xi]))\leq \Psi(\gamma(\bbm[\xi]^*_\gamma))=\Psi(\bbm[z])\leq C$$
with the constant $C$ independent of $\gamma.$ By taking the supremum for all the maps $\gamma\in {\cal F}$ we obtain the thesis.

\end{proof}

Then taking into account of Proposition \ref{loweresti} and \ref{loweresti2} the max-min inequality (P2) will follow once we have proved the next result. 

\begin{proposition}\label{butta} The following holds:\begin{equation}\label{rocco}\min_{\hbox{\scriptsize$\bbm[\xi]$}\in {\cal B}_0}\Psi(\bbm[\xi])=\min\left\{\Psi(\bbm[\xi])\,\Big|\,\bbm[\xi]\in {\cal B},\, \min_{j\neq k}d_g(\xi_j,\xi_k)=M^{-1}\right\} \to +\infty \hbox{ as }M\to +\infty.\end{equation}\end{proposition}

\begin{proof}
Let $\bbm[\xi]_n=(\xi_1^n,\ldots,\xi_N^n)\in {\cal B}$ be such that $\min_{j\neq k}d_g(\xi_j^n,\xi_k^n) \to 0$ as $n\to +\infty$. Possibly passing to a subsequence, we may assume 
\begin{equation}\label{cover}d_g(\xi_{j_0}^n,\xi_{k_0}^n) \to 0\hbox{ as }n\to +\infty\end{equation} for some  $j_0\neq k_0$. So, by using \eqref{iuni},  we may estimate
$$\Psi(\bbm[\xi]_n)=\frac{1}{2\pi}\sum_{j,k=1\atop j\neq k}^N\log\frac{1}{d_g(\xi_{j}^n,\xi_{k}^n)}+O(1) \geq \frac{1}{\pi}\log\frac{1}{d_g(\xi_{j_0}^n,\xi_{k_0}^n)}+O(1)\to +\infty.$$
\end{proof}

Hence, the  proof of property (P3) carried out in the next section allows us to conclude the proof of Proposition \ref{prop:ptocrit1} and Proposition \ref{prop:ptocrit2}. 

\subsection{Proof of (P3)}\label{section:P3}

We shall show that the compactness property (P3)  holds provided that $M$ is sufficiently large and  assumptions (H1), (H2), (H3), (H4), \eqref{acca5}-\eqref{acca55} hold. 

By Proposition \ref{loweresti}  and Proposition \ref{loweresti2} we get $\Psi^*=\Psi_M^*=O(1)$ as $M\to +\infty$. Then  (P3) will follow once we have proved the assertion of tangential derivative being non-zero over the boundary of ${\cal D}$ for uniformly bounded values of $\Psi$ provided that $M$ is large enough. We point out that  we will follow some argument of \cite{tea&pp}, where an analogous  compactness property is  proved  for positive $K$; however, unlike \cite{tea&pp}, here we have also to rule out  the possibility that some critical point occurs on the boundary $\partial \Sigma^+$ and this is a delicate situation  that needs to be handled carefully. 

We proceed by contradiction: assume that there exist $\bbm[\xi]_n=(\xi_1^n, \dots, \xi_N^n)\in {\cal M}^+$ and $(\beta^n_1, \beta_2^n)\neq (0,0)$ such that 
\beq\label{boupsi}
 \Phi(\bbm[\xi]_n)\to -\infty, \qquad
\Psi(\bbm[\xi]_n)=O(1),\eeq
  \beq\label{exp}\beta_1^n\nabla \Psi(\bbm[\xi]_n)+\beta_2^n\nabla \Phi(\bbm[\xi]_n)=0.\eeq Last expression can be read as $\nabla \Psi(\bbm[\xi]_n)$ and $\nabla \Phi(\bbm[\xi]_n)$ are linearly dependent. Observe that, according to the Lagrange multiplier Theorem,  this contradicts either the smoothness of $\partial {\cal D}$ or the nondegeneracy of $\nabla \Phi(\bbm[\xi]_n)$ on the tangent space at the level $\Psi^*$.
Without loss of generality we may assume \beq\label{ops}(\beta_1^n)^2+(\beta_2^n)^2=1\;\hbox{ and } \;\beta_1^n+\beta_2^n\geq 0.\eeq 
Observe that by \eqref{boupsi} $$2\sum_{j,k=1\atop j\neq k}^NG(\xi_j^n, \xi_k^n)=\Psi(\bbm[\xi]_n)-\Phi(\bbm[\xi]_n)\to +\infty,$$
which implies \beq\label{assoassu}\min_{j\neq k}d_g(\xi_j^n,\xi_k^n)=o(1).\eeq
Identity \eqref{exp} can be rewritten as 
\beq\label{leone}\frac{\beta_1^n+\beta_2^n}{4\pi}\frac{\nabla K(\xi_j)}{K(\xi_j)}
-(\beta_1^n+\beta_2^n)\sum_{i=1}^\ell\alpha_i\nabla_{\xi_j} G(\xi_j^n,p_i)+2(\beta_1^n-\beta_2^n)\sum_{k=1\atop k\neq j}^N\nabla_{\xi_j} G(\xi_j^n,\xi_k^n)=O(1)\quad \forall j.\eeq
The object of the remaining part of the section is to expand the left hand side of \eqref{exp} and to prove that the leading term is not zero, so that the contradiction arises. 
Before going on we fix some notation. For every $\xi\in\Sigma$ we introduce  normal coordinates $y_\xi$ from a neighborhood of $\xi$  onto $B_{r_0}(0)$ (the choice of $r_0$ is independent of $\xi$) which depend smoothly on $\xi \in \Sigma$.
Since $y_\xi(\xi)=0$ and $d_g(x,\xi)=|y_\xi(x)|$ for all $x \in y_\xi^{-1}(B_{r_0}(0))$, we have that
$$d_g(\xi_1,\xi_2)=|y_\xi(\xi_1)-y_\xi(\xi_2)|(1+o(1))$$ and 
\begin{equation} \label{1538}
\nabla_{\xi_1} \log d_g(\xi_1,\xi_2)= \frac{y_{\xi_2}(\xi_1)}{|y_{\xi_2}(\xi_1)|^2}=
\frac{y_\xi(\xi_1)-y_\xi(\xi_2)+o(|y_\xi(\xi_1)-y_\xi(\xi_2)|)}{|y_\xi(\xi_1)-y_\xi(\xi_2)|^2}
\end{equation}
as $\xi_1,\xi_2 \to \xi$. 

Hereafter we might pass to subsequences without further notice.

Let us split 
$\{1,\ldots,N\}=\tilde Z\cup Z_0\cup Z_1\cup\ldots \cup Z_\ell$ where $$\tilde Z=\{j\,|\, d_g(\xi_{j}^n, \partial \Sigma^+)\geq c\;\; \&\;\;d_g(\xi_j^n,p_i)\geq c \hbox{ for all }i\},$$
$$Z_0=\{j\,|\,d_g(\xi_{j}^n, \partial \Sigma^+)\to 0 \},\quad Z_i=\{j\,|\,\xi_j^n\to p_i\},\quad  i=1,\ldots,\ell.$$ 
We begin with the following three lemmas.
\begin{lemma}\label{step0}  $d_g(\xi_{j}^n,  \xi_k^n)\geq c $ for all $j,k\in Z_0$, $ j\neq k$. 
\end{lemma}
\begin{proof} Suppose by contradiction 
that there exists a point $\xi_0\in \partial\Sigma^+$ which is the limit of more than one sequence $\xi_j^n$; then define  the subset $Y_0\subset Z_0$ corresponding to such sequences:
$$Y_0:=\{j\,|\, \xi_j^n\to \xi_0\},\qquad d_g(\xi_j^n,\xi_0)\geq c\quad \forall j\in Z_0\setminus Y_0.$$ 
Let us  choose two indices $j_0, k_0\in Y_0$, $j_0\neq k_0$ in such a way that we may split $Y_0=I\cup (Y_0\setminus I)$ with\footnote{Here we use the notation $\sim$ to denote sequences which in the limit $n\to +\infty$ are of the same order. } 
$$j_0, k_0\in I,\;\;\;d_g(\xi_j^n, \xi_k^n)\sim d_g(\xi_{j_0}^n,\xi_{k_0}^n) \quad \forall j,k\in I,\, j\neq k.$$
 and \beq\label{oss0}d_g(\xi_{j_0}^n,\xi_{k_0}^n)=o(d_g(\xi_{j},\xi_{k}))\quad \forall j\in I, \forall k\in Y_0\setminus I.\eeq
Moreover, without loss of generality we may assume 
\beq\label{oss}d_g(\xi_j^n,  \partial \Sigma^+)\geq  d_g(\xi_{j_0}^n, \partial \Sigma^+)\quad \forall j\in I.\eeq
By \eqref{oss0} we have $$\nabla_{\xi_j}G(\xi_j^n, \xi_k^n)=o\bigg(\frac{1}{d_g(\xi_{j_0}^n,\xi_{k_0}^n)}\bigg)\quad \forall j\in I,\, k\in Y_0\setminus I.$$ Recalling assumption (H2)-(H3), by \eqref{oss} we derive 
\beq\label{ossa}\frac{|\nabla K(\xi_{j}^n)|}{K(\xi_{j}^n)}\sim\frac{1}{d_g(\xi_{j}^n,\partial\Sigma^+)}\leq C\frac{1}{d_g(\xi_{j_0}^n,\partial\Sigma^+)}\quad \forall j\in I.\eeq
Using the local chart 
 $y_{\xi_{j_0}^n}$, and recalling \eqref{1538}, the identities \eqref{leone} give
\beq\label{39}\begin{aligned}\frac{\beta_1^n+\beta_2^n}{4\pi}&\frac{\nabla K(\xi_{j}^n)}{K(\xi_{j}^n)}+2(\beta_1^n-\beta_2^n)\sum_{k\in I\atop k\neq j}\frac{y_{\xi_{j_0}^n}(\xi_j^n)-y_{\xi_{j_0}^n}(\xi_k^n)}{|y_{\xi_{j_0}^n}(\xi_j^n)-y_{\xi_{j_0}^n}(\xi_k^n)|^2}=o\bigg(\frac{\beta_1^n-\beta_2^n}{d_g(\xi_{j_0}^n,\xi_{k_0}^n)}\bigg)+O(1)\qquad \forall j\in I.\end{aligned}\eeq
Let us multiply the above identity by $y_{\xi_{j_0}^n}(\xi_j^n)$ and next sum in $j\in I$; 
taking into account  the following general  relation
\beq\label{ido}\sum_{j,k\in I\atop j\neq k} \frac{(z_j-z_k)(z_j-z)}{|z_j-z_k|^2}=\sum_{j,k\in I\atop j< k}1=\frac{\#I(\#I-1)}{2}\qquad \forall z_j, z\in\R^2\eeq
we obtain
$$
\frac{\beta_1^n+\beta_2^n}{4\pi}\sum_{j\in I}\frac{\nabla K(\xi_{j}^n)}{K(\xi_{j}^n)}y_{\xi_{j_0}^n}(\xi_j^n)+(\beta_1^n-\beta_2^n)\#I(\#I-1)=
o(\beta_1^n-\beta_2^n)+O(d_g(\xi_{j_0}^n,\xi_{k_0}^n))$$ 
by which, recalling that $\# I\geq 2$ (since $j_0, k_0\in I$) and using \eqref{ossa},
 we get \beq\label{beta120}\beta_1^n-\beta_2^n=(\beta_1^n+\beta_2^n)O\Big(\frac{d_g(\xi_{j_0}^n,\xi_{k_0}^n)}{d_g(\xi_{j_0}^n,\partial\Sigma^+)}\Big)+O(d_g(\xi_{j_0}^n,\xi_{k_0}^n)).\eeq 
Next let us multiply \eqref{39}   by $\nu_j^n:= \frac{\nabla K(\xi_{j}^n) }{| \nabla K(\xi_{j}^n)|}
$ and  sum in $j\in I$:
\beq\label{sche}\begin{aligned}\frac{\beta_1^n+\beta_2^n}{4\pi}&\sum_{j\in I}\frac{|\nabla K(\xi_{j}^n) |}{  K(\xi_{j}^n)}+2(\beta_1^n-\beta_2^n)\sum_{j, k\in I\atop k\neq j}\frac{y_{\xi_{j_0}^n}(\xi_j^n)-y_{\xi_{j_0}^n}(\xi_k^n)}{|y_{\xi_{j_0}^n}(\xi_j^n)-y_{\xi_{j_0}^n}(\xi_k^n)|^2} \nu_j^n=o\bigg(\frac{\beta_1^n-\beta_2^n}{d_g(\xi_{j_0}^n,\xi_{k_0}^n)}\bigg)+O(1).\end{aligned}\eeq
Since   $K$ is of class $C^{2}(\Sigma)$ according to assumption (H2), we deduce
$$\nu_j^n- \nu_k^n=O(d_g(\xi_{j}^n,\xi_{k}^n))$$ by which we can write
$$\sum_{j, k\in I\atop k\neq j}\frac{y_{\xi_{j_0}^n}(\xi_j^n)-y_{\xi_{j_0}^n}(\xi_k^n)}{|y_{\xi_{j_0}^n}(\xi_j^n)-y_{\xi_{j_0}^n}(\xi_k^n)|^2} \nu_j^n=\sum_{j, k\in I\atop j<k}\frac{y_{\xi_{j_0}^n}(\xi_j^n)-y_{\xi_{j_0}^n}(\xi_k^n)}{|y_{\xi_{j_0}^n}(\xi_j^n)-y_{\xi_{j_0}^n}(\xi_k^n)|^2} (\nu_j^n-\nu_k^n)=O(1).$$ 
By inserting the above estimate  into \eqref{sche}   and recalling \eqref{ossa} we arrive at  
$$\frac{\beta_1^n+\beta_2^n}{4\pi}\frac{1}{d_g(\xi_{j_0}^n,\partial\Sigma^+)}=
o\bigg(\frac{\beta_1^n-\beta_2^n}{d_g(\xi_{j_0}^n,\xi_{k_0}^n)}\bigg)+
O(1)$$ 
or, equivalently, \beq\label{beta122}\beta_1^n+\beta_2^n=(\beta_1^n-\beta_2^n)o\bigg(\frac{d_g(\xi_{j_0}^n,\partial\Sigma^+)}{d_g(\xi_{j_0}^n,\xi_{k_0}^n)}\bigg)+
O(d_g(\xi_{j_0}^n,\partial\Sigma^+)).\eeq Combining \eqref{beta120} and \eqref{beta122} we get $\beta_1^n-\beta_2^n=o(1), \beta_1^n+\beta_2^n=o(1)$, in contradiction with \eqref{ops}.

\end{proof}
\begin{lemma}\label{step1}
The following holds:
\begin{enumerate} 
 \item[a)] if $\#Z_i=1$ for some $i=1,\ldots,\ell$, then $\beta_1^n+\beta_2^n\to 0$;
 \item[b)] if $Z_0\neq\emptyset$,  then $\beta_1^n+\beta_2^n \to 0$;
\item[c)]  if $d_g(\xi_j^n,\xi_k^n)=o(1)$ for some $j,k\in \tilde Z$, $j\neq k$, then $\beta_1^n-\beta_2^n\to 0;$
\item[d)]  there exists $i\in\{1,\ldots,\ell\}$ such that $\#Z_i\geq 2. $\end{enumerate}
\end{lemma}
\begin{proof} 
Assume that $i=1,\ldots,\ell$ is such that 
$\#Z_{i}=1,$ say $Z_{i}=\{j_0\}$. Then, using the local chart $y_{p_{i}}$,  the identities \eqref{leone} give
$$(\beta_1^n+\beta_2^n)\alpha_{i}\frac{y_{p_{i}}(\xi_{j_0}^n)}{|y_{p_{i}}(\xi_{j_0}^n)|^2}=O(1)$$
which implies $\beta_1^n+\beta_2^n=o(1)$, then $a)$ follows. 

Similarly, assume that $Z_0\neq\emptyset$ and let $j_0\in Z_0$. According to Lemma \ref{step0} we have $d_g( \xi_{j_0}^n, \xi_j)\geq c$ for all $j\neq j_0$. In this case the identity \eqref{leone} with $j=j_0$ becomes 
$$(\beta_1^n+\beta_2^n)\frac{\nabla K(\xi_{j_0}^n)}{ K(\xi_{j_0}^n)}=O(1).$$ According to assumption (H2)-(H3) we have $\frac{|\nabla K(\xi_{j_0}^n)|}{K(\xi_{j_0}^n)}\sim\frac{1}{d_g(\xi_{j_0}^n,\partial\Sigma^+)}$, and  b) is thus established.
 
Next, suppose that $j_0, k_0\in Z_0$ with $j_0\neq k_0$ are such that  $d_g(\xi_{j_0}^n,\xi_{k_0}^n)=o(1)$. We may assume
$$d_g(\xi_{j_0}^n,\xi_{k_0}^n)=\min_{j,k\in Z_0,j\neq k}d_g(\xi_j^n,\xi_k^n)\quad \forall n\in\N.$$
So we can split $Z_0=I\cup J$ where $$I=\big\{j\in Z_0\,\big|\, d_g(\xi_{j}^n,\xi_{j_0}^n)\sim d_g(\xi_{j_0}^n,\xi_{k_0}^n)\big\}\cup\{j_0\},\qquad J=\big\{j\in Z_0\,\big|\,  d_g(\xi_{j_0}^n,\xi_{k_0}^n)=o(d_g(\xi_j^n,\xi_{j_0}^n))\big\}.$$
We observe that by construction 
$$d_g(\xi_{j_0}^n,\xi_{k_0}^n)\sim d_g(\xi_j^n,\xi_k^n)\quad \forall j,k\in I,\,j\neq k$$ and
$d_g(\xi_{j_0}^n,\xi_{k_0}^n)=o(d_g(\xi_j^n,\xi_k^n))$ for all $j\in I$ and $k\in J$, by which $$\nabla_{\xi_j}G(\xi_j^n, \xi_k^n)=o\bigg(\frac{1}{d_g(\xi_{j_0}^n,\xi_{k_0}^n)}\bigg)\quad \forall j\in I,\, k\in J.$$
Then for any $j\in I$, using the local chart $y_{\xi_{j_0}^n}$ the identities \eqref{leone} give
\beq\label{qqw100}(\beta_1^n-\beta_2^n) \sum_{k\in I\atop k\neq j} \frac{y_{\xi_{j_0}^n}(\xi_{j}^n)-y_{\xi_{j_0}^n}(\xi_k^n)}{|y_{\xi_{j_0}^n}(\xi_j^n)-y_{\xi_{j_0}^n}(\xi_k^n)|^2}=o\bigg(\frac{1}{d_g(\xi_{j_0}^n,\xi_{k_0}^n)}\bigg)\quad \forall j\in I.\eeq So we multiply \eqref{qqw100} by $y_{\xi_{j_0}^n}(\xi_j^n)$ and sum in $j\in I$: by using \eqref{ido} we arrive at 
$$(\beta_1^n-\beta_2^n)\#I(\#I-1)
=o(1) .$$ Taking into account that $I$ has at least two elements, since  $j_0,\, k_0\in I$, we deduce  $\beta_1^n-\beta_2^n=o(1)$, and $\rm c)$ follows. 

Finally  assume by contradiction that  $Z_{i}$ consists of at most one index for every $i=1,\ldots,\ell$. Then, combining this with Lemma \ref{step0} and  \eqref{assoassu} we find that $d_g(\xi_j^n,\xi_k^n)=o(1)$ for some $j,k\in \tilde Z$, $j\neq k$. Then part c) gives $\beta_1^n-\beta_2^n=o(1)$; consequently $\beta_1^n+\beta_2^n\geq c$ by \eqref{ops}, and part b) implies $Z_0=\emptyset$, that is $K(\xi_j^n)=O(1)$ for all $j$. So, thanks to \eqref{boupsi} we deduce 
$$2\sum_{i=1}^\ell \alpha_i\sum_{j=1}^N  G(\xi_j, p_i)= -\Phi(\bbm[\xi]_n)-\Phi(\bbm[\xi]_n)+O(1)\to +\infty.$$
which implies $$\min_{i,j} d_g(\xi_j^n, p_i)=o(1).$$
Then   $Z_{i}$ is nonempty for some $i=1,\ldots,\ell$, so $\#Z_{i}=1$ for such $i$. 
Then, by $\rm a)$ we derive $\beta_1^n+\beta_2^n=o(1)$, and the contradiction arises.

\end{proof}

\begin{lemma} \label{step2}
If $i=1,\ldots,\ell$ is such that  $\#Z_i\geq 2$, then $d_g(\xi^n_{j},p_{i})=O(d_g(\xi^n_{j},\xi_{k}^n))$ for all  $j,k\in Z_i$, $j\neq k$. 
 \end{lemma}
\begin{proof} 
Fix $i=1,\ldots,\ell$  with $\#Z_i\geq 2$. We proceed by contradiction assuming  that there exist indices  $j_0, k_0\in Z_i$, $j_0\neq k_0$, such that \beq\label{mindi}\frac{d_g(\xi^n_{j_0},\xi_{k_0}^n)}{d_g(\xi^n_{j_0},p_{i})}=\min_{i,j\in Z_i,j\neq j}\frac{d_g(\xi^n_{j},\xi_{k}^n)}{d_g(\xi^n_{j},p_{i})}\to 0.\eeq
 According to \eqref{mindi}  we can split $Z_i=I\cup J$ where $$I=\big\{j\in Z_i\,\big|\, d_g(\xi_{j}^n,\xi_{j_0}^n)\sim d_g(\xi_{j_0}^n,\xi_{k_0}^n)\big\}\cup\{j_0\},\qquad J=\big\{j\in Z_i\,\big|\,  d_g(\xi_{j_0}^n,\xi_{k_0}^n)=o(d_g(\xi_j^n,\xi_{j_0}^n))\big\}.$$ Clearly $j_0,k_0\in I$, so $\#I\geq 2$.
We observe that by construction 
$$d_g(\xi_j^n, p_{i})\sim d_g(\xi_{j_0}^n,p_{i})\quad \forall j\in I$$ and $$d_g(\xi_j^n, \xi_k^n)\sim d_g(\xi_{j_0}^n,\xi_{k_0}^n)\quad \forall j,k\in I,\; j\neq k.$$ 
Moreover $d_g(\xi_{j_0}^n,\xi_{k_0}^n)=o(d_g(\xi_j^n,\xi_k^n))$ for all $j\in I$ and $k\in J$, by which $$\nabla_{\xi_j}G(\xi_j^n, \xi_k^n)=o\bigg(\frac{1}{d_g(\xi_{j_0}^n,\xi_{k_0}^n)}\bigg)\quad \forall j\in I,\, k\in J.$$
Using the local chart 
 $y_{\xi_{j_0}^n}$, and recalling \eqref{1538}, the identities \eqref{leone} give
\beq\label{mash}\begin{aligned}(\beta_1^n+\beta_2^n)&\alpha_i\frac{y_{\xi_{j_0}^n}(\xi_j^n)-y_{\xi_{j_0}^n}(p_{i})}{|y_{\xi_{j_0}^n}(\xi_j^n)-y_{\xi_{j_0}^n}(p_{i})|^2}-2(\beta_1^n-\beta_2^n)\sum_{k\in I\atop k\neq j}\frac{y_{\xi_{j_0}^n}(\xi_j^n)-y_{\xi_{j_0}^n}(\xi_k^n)}{|y_{\xi_{j_0}^n}(\xi_j^n)-y_{\xi_{j_0}^n}(\xi_k^n)|^2}\\ &=o\Big(\frac{\beta_1^n+\beta_2^n}{d_g(\xi_{j_0}^n,p_{i})}\Big)+o\bigg(\frac{\beta_1^n-\beta_2^n}{d_g(\xi_{j_0}^n,\xi_{k_0}^n)}\bigg)+O(1)\qquad \forall j\in I.\end{aligned}\eeq
Let us multiply the above identity by $y_{\xi_{j_0}^n}(\xi_j^n)$ and next sum in $j\in I$;  using \eqref{ido} 
we obtain
$$
(\beta_1^n -\beta_2^n)\#I(\#I-1)=
(\beta_1^n+\beta_2^n)O\Big(\frac{d_g(\xi_{j_0}^n,\xi_{k_0}^n)}{d_g(\xi_{j_0}^n,p_{i})}\Big)+o(\beta_1^n-\beta_2^n)+O(d_g(\xi_{j_0}^n,\xi_{k_0}^n)).$$ 
by which we get \beq\label{beta12}\beta_1^n-\beta_2^n=
O\Big(\frac{d_g(\xi_{j_0}^n,\xi_{k_0}^n)}{d_g(\xi_{j_0}^n,p_{i})}\Big)=o(1). 
\eeq 
Next we multiply  identity \eqref{mash} by $y_{\xi_{j_0}^n}(\xi_j^n)-y_{\xi_{j_0}^n}(p_{i})$ and sum in $j\in I$; by using again relation \eqref{ido}
we get
$$(\beta_1^n+\beta_2^n)\alpha_i\#I=
(\beta_1^n -\beta_2^n)\#I(\#I-1)+
o(\beta_1^n+\beta_2^n)+(\beta_1^n-\beta_2^n)o\Big(\frac{d_g(\xi_{j_0}^n,p_{i})}{d_g(\xi_{j_0}^n,\xi_{k_0}^n)}\Big)+O(d_g(\xi_{j_0}^n,p_{i})).$$ 
Taking inti account of  \eqref{beta12} we obtain that  $\beta_1^n+\beta_2^n=o(1)$, in contradiction with  \eqref{ops}.

\end{proof}
\medskip  

Let us sum up all the previous information contained in the Lemma \ref{step0}, Lemma \ref{step1} and Lemma \ref{step2} in order to finally get the conclusion. 
According to $\rm d)$ of Lemma \ref{step1}, there exists $i=1,\ldots,\ell$ be such that  $\#Z_i \geq 2$. Let us split $Z_i$   as $Y_1\cup \dots \cup Y_l$, $l\geq 1$, in such a way that
$$d_g(\xi_j^n,p_i)\sim d_g(\xi_k^n,p_i)\quad \forall j,k\in Y_r$$ 
and 
\beq\label{mario2}d_g(\xi_j^n,p_i)=o(d_g(\xi_k^n,p_i))\quad \forall j\in Y_r, \, \forall k\in Y_{r+1}\cup\dots \cup Y_l.\eeq 
Notice that by construction $d_g(\xi_j^n,\xi_k^n)\sim d_g(\xi_k^n,p_i)$ for all $j\in Y_r$, $k\in Y_{r+1}\cup\dots \cup Y_l$, and by Lemma \ref{step2} $d_g(\xi_j^n,\xi_k^n)\sim  d_g(\xi_k^n,p_i)$ for all $j, k\in Y_r$, $j\neq k$, yielding 
\beq\label{mario3}d_g(\xi_j^n,\xi_k^n)\sim d_g(\xi_k^n,p_i)\quad \forall j\in Y_r,\;\;\forall k\in Y_{r}\cup\dots \cup Y_{l} ,\quad j\neq k.\eeq
Combining \eqref{mario2}-\eqref{mario3} we get
$$\nabla_{\xi_j}G(\xi_j^n, \xi_k^n)=o\bigg(\frac{1}{d_g(\xi_j^n,p_i)}\bigg)\quad \forall j\in Y_r, \,\,\forall k\in Y_{r+1}\cup \dots \cup Y_{l}.$$
Next consider $r\in\{1,\ldots,l\}$ and write \eqref{leone}  for $j\in Y_r$ in the local chart $y_{p_i}$: 
\beq\label{qw2}2(\beta_1^n-\beta_2^n) \!\!\sum_{k\in Y_1\cup \dots \cup Y_r\atop k\neq j}\!\! \frac{y_{p_i}(\xi_j^n)-y_{p_i}(\xi_k^n)}{|y_{p_i}(\xi_j^n)-y_{p_i}(\xi_k^n)|^2}=(\beta_1^n+\beta_2^n)\alpha_i \frac{y_{p_i}(\xi_j^n)}{|y_{p_i}(\xi_j^n)|^{2}}+o\bigg(\frac{1}{d_g(\xi_j^n,p_i)}\bigg)\eeq 
for all $j\in Y_r$. By \eqref{mario2} and \eqref{mario3} we find  $|y_{p_i}(\xi_j^n)|=o(|y_{p_i}(\xi_j^n)-y_{p_i}(\xi_k^n)|) $ for all $ j\in Y_1\cup \dots \cup Y_{r-1}$, $k\in Y_r$, and we can compute 
$$\frac{\langle y_{p_i}(\xi_j^n)-y_{p_i}(\xi_k^n),y_{p_i}(\xi_j^n)\rangle }{|y_{p_i}(\xi_j^n)-y_{p_i}(\xi_k^n)|^2}\!=1+\frac{\langle y_{p_i}(\xi_j^n)-y_{p_i}(\xi_k^n),y_{p_i}(\xi_k^n)\rangle }{|y_{p_i}(\xi_j^n)-y_{p_i}(\xi_k^n)|^2}\!=1+o(1)$$
for all $j\!\in\! Y_r$ and $k\!\in\! Y_1\cup\dots \cup Y_{r-1}$. Combining this with identity \eqref{ido}, 
by taking the inner product of \eqref{qw2} with $y_{p_i}(\xi_j^n)$ and summing up in $j\in Y_r$ we get that
\beq\label{gra2}2(\beta_1^n-\beta_2^n)\bigg(
\frac{\# Y_r(\# Y_r-1)}{2}+\#Y_r(\#Y_1+\ldots+\#Y_{r-1})
\bigg)=(\beta_1^n+\beta_2^n)\alpha_i\# Y_r +o(1)\quad \forall r\geq 1.\eeq
Since $\# Z_{i}\geq 2$, notice that the coefficient in brackets on the left hand side of  \eqref{gra2} is positive when  $r=l$, and then $\frac{1}{\alpha_i}(\beta_1^n-\beta_2^n)$ and $\beta_1^n+\beta_2^n$ are positively proportional up to higher order terms. So, by \eqref{ops} and \eqref{gra2} (with $r=l$) we deduce that
\beq\label{bell0}\frac{\beta_1^n-\beta_2^n}{\alpha_i},\, \beta_1^n+\beta_2^n\geq c>0.\eeq 
  By \eqref{gra2} 
we also have
\beq\label{bell}|\beta_2^n|\geq c>0.\eeq Indeed, if $\beta_2^n =o(1)$ we would obtain $\beta_1^n=1+o(1)$ and, consequently,   $\#Y_1-1=\alpha_i$
in view of \eqref{gra2} (with  $r=1$), contradicting the compactness assumption \eqref{acca5}--\eqref{acca55}. Let us evaluate the different pieces of the energy as follows: 
$$\begin{aligned}&\sum_{j,k\in Z_i \atop j\neq k}   G(\xi_j^n,\xi_k^n)-\alpha_i\sum_{j\in Z_i} G(\xi_j^n,p_i)\\&=\sum_{r=1}^l\sum_{j,k\in Y_r \atop j\neq k}   G(\xi_j^n,\xi_k^n)+2\sum_{r=1}^l\!\!\sum_{j\in Y_r\atop k\in Y_1\cup \dots \cup Y_{r-1}}\!\!\!\!\!\!  G(\xi_j^n,\xi_k^n)-\alpha_i\sum_{r=1}^l \sum_{j\in Y_r}  G(\xi_j^n,p_i)\\
&=-\frac{1}{2\pi} \sum_{r=1}^l \!\bigg[\sum_{j,k\in Y_r\atop j\neq k}   \log d_g(\xi_j^n,\xi_k^n) +2 \!\!\!\!\sum_{j\in Y_r\atop k\in Y_1\cup\dots\cup Y_{r-1}}\!\!\!\!\!\!   \log d_g(\xi_j^n,\xi_k^n)-\alpha_i \sum_{j\in Y_r} \log d_g(\xi_j^n,p_i)\bigg]+O(1).
\end{aligned}$$ 
Since $d_g(\xi_j^n,\xi_k^n)\sim d_g(\xi_j^n,p_i)$ for all $j\in Y_r$ and $k\in Y_1\cup \dots \cup Y_{r}$ with $j\neq k$ thanks  to \eqref{mario3}, fixed $j_r\in Y_r$ 
 we have that
\begin{eqnarray*}
&&\sum_{j,k\in Y_r\atop j\neq k}   \log d_g(\xi_j^n,\xi_k^n)+2\sum_{j\in Y_r\atop k\in Y_1\cup\dots\cup Y_{r-1}}  \log d_g(\xi_j^n,\xi_k^n)-\alpha_i\sum_{j\in Y_r} \log d_g(\xi_j^n,p_i)\\ 
&&=  \bigg(\sum_{j,k\in Y_r\atop j\neq k} 1  +2\sum_{j\in Y_r\atop k\in Y_1\cup\dots\cup Y_{r-1}}1   -\alpha_i\sum_{j\in Y_r}1 \bigg)\log d_g(\xi_{j_r}^n,p_i)+O(1)\\ 
&&=\#Y_r\Big(\#Y_{r}-1   +2(\#Y_1+\ldots+\#Y_{r-1} ) -\alpha_i \Big)\log d_g(\xi_{j_r}^n,p_i) +O(1)\\
&&= \frac{2\beta_2^n+o(1)}{\beta_1^n-\beta_2^n}\alpha_i\#Y_r
\log d_g(\xi_{j_r}^n,p_i) +O(1)
\end{eqnarray*}
where in the last identity we have used \eqref{gra2}.  Recalling \eqref{bell0}--\eqref{bell}, we have thus proved that 
\beq\label{puff}\frac{1}{\beta_2^n}\bigg(\sum_{j,k\in Z_i \atop j\neq k}    G(\xi_j^n,\xi_k^n)-\alpha_i\sum_{j\in Z_i} G(\xi_j^n,p_i)\bigg)\to +\infty\eeq 
for all $Z_i $ with $\#Z_i\geq 2$. 
On the other hand by \eqref{bell0} and part $\rm a)$ of Lemma \ref{step1} we have that for any $i=1,\ldots,\ell$ either $Z_i=\emptyset$ or $\#Z_i\geq 2$. Moreover part $\rm b)$  and c) 
 give $K(\xi_j^n)=O(1)$ for all $j$ and   $G(\xi_j^n,\xi_k^n)=O(1)$ for all $(j,k)\notin  \bigcup_{i=1}^\ell (Z_i \times Z_i)$. Then we conclude 
\begin{eqnarray*}
\frac{1}{\beta_2^n} \Psi(\bbm[\xi]_n)&=&\frac{1}{\beta_2^n}\bigg[\sum_{j,k=1\atop j\neq k}^N   G(\xi_j^n, \xi_k^n)-\sum_{i=1}^\ell\alpha_i\sum_{j=1}^N G(\xi_j^n,p_i)\bigg]+O(1) \\ 
&=& \frac{1}{\beta_2^n} \sum_{i=1}^\ell \bigg[\sum_{j,k\in Z_i \atop j\neq k}     G(\xi_j^n, \xi_k^n)- \alpha_i\sum_{j\in Z_i}  G(\xi_j^n,p_i)\bigg]+O(1) \to +\infty
\end{eqnarray*}
in view of \eqref{puff}, in contradiction with \eqref{boupsi} and \eqref{bell}.
 
\

\section{More existence results}\label{Section:ultima}

In this section we get other existence results for solutions to  \ref{liouvgeom}, using Proposition \ref{prop:espofigue}.

\

In general it is hard to guarantee the validity of condition \emph{b)} of Proposition \ref{prop:espofigue} and in all our previous results we got it for $K$ sign-changing and satisfying in particular condition \eqref{condiz}  discussed in Remark \ref{rem:A<0}, which implies $\mathcal  A(\bbm[\xi])<0$ for all $\bbm[\xi]\in\cal M$;  in such a case condition \emph{b)} is satisfied and this led then to solutions to \ref{liouvgeom}, for $\rho_{geo}$ close to integer multiples of $8\pi$ \emph{from the left} hand side.

\

As noticed in Remark \ref{rem:A>0} it is not possible 
to impose on $K$ a simple  condition  like \eqref{condiz} and have instead $\mathcal  A(\bbm[\xi])>0$ for all $\bbm[\xi]\in\cal M$, which would lead then to solutions to \ref{liouvgeom}, for $\rho_{geo}$ close to integer multiples of $8\pi$ \emph{from the right} hand side. And this holds both for $K$ sign-changing or positive.

\

Moreover observe that if we consider functions $K>0$ then, since $\Delta_g \log K$ changes sign on $\Sigma$,  it is not even possible to impose on it the simple condition \eqref{condiz} in Remark \ref{rem:A<0} and  have $\mathcal  A(\bbm[\xi])<0$ for all $\bbm[\xi]\in\cal M$. Namely for $K$ positive it is hard to get solutions to \ref{liouvgeom} via Proposition \ref{prop:espofigue} even for $\rho_{geo}$ close to integer multiples of $8\pi$ \emph{from the left} hand side.

\

Nevertheless in this section we exhibit classes of functions $K$, sign-changing or also positive, (and values of $\alpha_i$'s) for which we are able to produce a stable critical point $\bbm[\xi]^*_{\underline{\alpha}}$ of $\Psi_{\underline{\alpha}}$ fulfilling conditions \emph{b)} and \emph{c)} of Proposition \ref{prop:espofigue}, obtaining in this way solutions to \ref{liouvgeom}, for $\rho_{geo}$ close to an integer multiple of $8\pi$ {\it both from the right and from the left}.

\

Let us first state rigorously these results on any compact orientable surface without boundary $(\Sigma,g)$. Next we will deduce from them Theorem \ref{thm:examples1}, Theorem \ref{thm:examples2} and Theorem \ref{thm:examples3} in the introduction which are related to the case of the standard sphere and to situations for which general existence results  are not available in the literature.\\\\
\
Let $m, N\in\mathbb N$, $p_1,\ldots,p_m\in\Sigma$ and   $-1<\alpha_{\star}\leq \alpha^{\star}$ be fixed. For any $\underline s=(s_1,\ldots, s_m)$, $\alpha_*\leq s_i\leq\alpha^*$, we define the functional 
\begin{equation}\label{def:Dalpha}
D_{\underline{s}}(\bbm[\xi]):=\sum_{j=1}^N h(\xi_j, \xi_j)+\frac{1}{4\pi}\sum_{j=1}^N f_g(\xi_j)-\sum_{i=1}^m s_i\sum_{j=1}^N G(\xi_j,p_i)+\sum_{j,k=1\atop j\neq k}^NG(\xi_j,\xi_k),
\end{equation}which is well defined in the set
\beq\label{def:dom}{\cal M}:=(\Sigma\setminus\{p_1,\dots,p_m\})^N\setminus \Delta,\qquad \Delta:=\Big\{\bbm[\xi]\in \Sigma^N\,\Big|\,  \xi_j= \xi_k\,\hbox{ for some } j\neq k\Big\}\eeq
where 
$G(x,p)$ is  the Green's function of $-\Delta_g$ over $\Sigma$ with singularity at $p$, the function $h$ denotes  its regular part as in \eqref{0948}, and $f_g$ is defined by \eqref{delgel}. Next we fix $\bbm[\bar\xi]\in{\cal M}$ and 
 we consider a radius $r=r(\bbm[{\bar\xi}])>0$ such that 
 \begin{equation}\label{extra}
 B_{2r}(\bbm[\bar\xi]):=\{\bbm[\xi]\in\Sigma^N\,|\,d_g(\bbm[\bar\xi],\bbm[\xi])<2r \}\subset {\cal M}.
 \end{equation}
where in the above brackets with a small abuse of notation we have continued to denote by $d_g$ the distance on  $\Sigma^N$. 
Let us set 
\[
M^+_{\bar\xi,\alpha_*,\alpha^*}:=\!\!\max\limits_{{\tiny\begin{array}{c}\alpha_*\leq s_i\leq\alpha^*\\i=1,\ldots,m\end{array}}}\max\limits_{\hbox{\scriptsize$\bbm[\xi]$}\in \overline{B_{\frac r2}(\hbox{\scriptsize$\bbm[\bar\xi]$})}}D_{\underline{s}}(\bbm[\xi])\qquad
\mbox{and}
\qquad 
m^+_{\bar\xi,\alpha_*,\alpha^*}:=\!\!\min\limits_{{\tiny\begin{array}{c}\alpha_*\leq s_i\leq\alpha^*\\i=1,\ldots,m\end{array}}}\min\limits_{\hbox{\scriptsize$\bbm[\xi]$}\in\partial B_r(\hbox{\scriptsize$\bbm[\bar\xi]$})} D_{\underline{s}}(\bbm[\xi]).
\]
Let  $K\in C^2(\Sigma)$ be  either a  positive  function or a sign-changing function satisfying the following 
\begin{equation}\label{prima}
\inf_{\hbox{\scriptsize$\bbm[\xi]$}\in B_r(\hbox{\scriptsize$\bbm[\bar\xi]$})} \min_{j=1,\ldots,N} K(\xi_j)>0,
\end{equation}
\begin{equation}\label{seconda}
\max_{\hbox{\scriptsize$\bbm[\xi]$}\in \overline{B_{\frac r2}(\hbox{\scriptsize$\bbm[\bar\xi]$})}} \Big(\frac{1}{4\pi}\sum_{j=1}^N \log K(\xi_j)\Big)< \min_{\hbox{\scriptsize$\bbm[\xi]$}\in {\partial B_{r}(\hbox{\scriptsize$\bbm[\bar\xi]$})}}\Big(\frac{1}{4\pi}\sum_{j=1}^N \log K(\xi_j)\Big)+m^+_{\bar\xi,\alpha_*,\alpha^*}-M^+_{\bar\xi,\alpha_*,\alpha^*}
\end{equation}
and 
\begin{equation}\label{terza}
\inf_{\hbox{\scriptsize$\bbm[\xi]$}\in {B_{r}(\hbox{\scriptsize$\bbm[\bar\xi]$})}}\min_{j\in1,\ldots,N} \Delta_{g}\log K(\xi_j)\geq \frac{1}{|\Sigma|}.
\end{equation}

Finally let us define the  class of functions  \begin{eqnarray}\label{classeDiK}\mathcal  K_{\bar\xi,\alpha_*,\alpha^*}^+:=\left\{K\in C^2(\Sigma)\ :\   K \mbox{ satisfies  \eqref{prima}, \eqref{seconda}  and \eqref{terza}}\right\}\end{eqnarray}

It is clear that for any fixed $\bbm[\bar\xi], \alpha_*,\alpha^*$ we can exhibit  plenty of functions $K$, both positive or sign-changing,  belonging to the class $\mathcal K_{\bar\xi,\alpha_*,\alpha^*}^+$. Roughly speaking,  \eqref{prima}, \eqref{seconda}  and \eqref{terza} are satisfied if $K$ has   \emph{sufficiently convex} local minima at  points $\bar\xi_1,\ldots, \bar\xi_N, $ where $\bbm[\bar\xi]=(\bar \xi_1,\ldots, \bar\xi_N)$. 
\

\begin{theorem}\label{prop:example1}
For any $K\in\mathcal  K^+_{\bar\xi,\alpha_*,\alpha^*}$ there exists $\delta=\delta(\alpha_*,\alpha^*,K)>0$ such that if
\begin{equation}\label{hpsum}
\sum_{i=1}^m \alpha_i  =  2N-{\chi(\Sigma)}+\frac{\varepsilon}{4\pi}\qquad  {\mbox{ for }\varepsilon\in (0, \delta)} \end{equation}
and moreover
\begin{equation}\label{alphatra}\alpha_*\leq\alpha_i\leq\alpha^*\qquad\mbox{ for all $i=1,\ldots, m$}
\end{equation}
then \ref{liouvgeom} admits a solution {with $\rho_{geo}=8\pi N +\varepsilon$.}
\end{theorem}

\begin{proof} 
By \eqref{extra} and \eqref{prima} we deduce 
\[
\overline{B_r(\bbm[\bar\xi])}\subset \mathcal M^+=(\Sigma^+\setminus\{p_1,\ldots,p_m\})^N\setminus\Delta
\]
and hence the functional $\Psi_{\underline{\alpha}}$, introduced in \eqref{psi0}, is well defined in $ \overline{B_r(\bbm[\bar\xi])}$.\\
We rewrite $\Psi_{\underline{\alpha}}$ as
\begin{equation}
\label{Psi=D+log}
\Psi_{\underline{\alpha}}(\bbm[\xi])= D_{\underline{\alpha}}(\bbm[\xi])+\frac{1}{4\pi}\sum_{j=1}^N \log K(\xi_j),
\end{equation}
where $D_{\underline{\alpha}}$ is  defined in \eqref{def:Dalpha} with $\underline{s}=\underline{\alpha}=(\alpha_1,\ldots,\alpha_m)$.
For the sake of clarity we split the remaining part of the proof into three steps.

\

\noindent STEP 1. {\it We show that assumptions \eqref{seconda} and \eqref{alphatra} imply the existence of a local minimum point (and so a stable critical point) $\bbm[\xi]^*_{\underline\alpha}\in B_r (\bbm[\bar\xi])$   of $\Psi_{\underline{\alpha}}$.} 
\\
\\
In order to prove this step  it is enough to show that
$$
\max\limits_{\hbox{\scriptsize$\bbm[\xi]$}\in \overline{B_{\frac r2}(\hbox{\scriptsize$\bbm[\bar\xi]$})}}\Psi_{\underline{\alpha}}(\bbm[\xi])
<
\min\limits_{\hbox{\scriptsize$\bbm[\xi]$}\in {\partial B_{r}(\hbox{\scriptsize$\bbm[\bar\xi]$})}}\Psi_{\underline{\alpha}}(\bbm[\xi]).
$$
Indeed, by virtue of \eqref{seconda} and assumption \eqref{alphatra} on  the $\alpha_i$'s  we compute
\begin{eqnarray*}
 \max\limits_{\hbox{\scriptsize$\bbm[\xi]$}\in \overline{B_{\frac r2}(\hbox{\scriptsize$\bbm[\bar\xi]$})}}\Psi_{\underline{\alpha}}(\bbm[\xi])
& \stackrel{\eqref{Psi=D+log}}{\leq } &
\max\limits_{\hbox{\scriptsize$\bbm[\xi]$}\in \overline{B_{\frac r2}(\hbox{\scriptsize$\bbm[\bar\xi]$})}}  D_{\underline{\alpha}}(\bbm[\xi])+ \frac{1}{4\pi}\max\limits_{\hbox{\scriptsize$\bbm[\xi]$}\in \overline{B_{\frac r2}(\hbox{\scriptsize$\bbm[\bar\xi]$})}}\Big( \sum_{j=1}^N \log K(\xi_j) \Big)
\\
&\stackrel{\eqref{alphatra}}{\leq}  &
M^+_{\bar\xi,\alpha_*,\alpha^*}+\frac{1}{4\pi}\max\limits_{\hbox{\scriptsize$\bbm[\xi]$}\in \overline{B_{\frac r2}(\hbox{\scriptsize$\bbm[\bar\xi]$})}}\Big(\sum_{j=1}^N\log K(\xi_j)\Big)
\\
&\stackrel{\eqref{seconda}}{<} &
m^+_{\bar\xi,\alpha_*,\alpha^*} +\frac{1}{4\pi}\min_{\hbox{\scriptsize$\bbm[\xi]$}\in {\partial B_{r}(\hbox{\scriptsize$\bbm[\bar\xi]$})}}\Big(\sum_{j=1}^N \log K(\xi_j)\Big)
\\
&\stackrel{\eqref{alphatra}}{\leq} &
\min\limits_{\hbox{\scriptsize$\bbm[\xi]$}\in {\partial B_{r}(\hbox{\scriptsize$\bbm[\bar\xi]$})}}D_{\underline{\alpha}}(\bbm[\xi]) + \frac{1}{4\pi} \min\limits_{\hbox{\scriptsize$\bbm[\xi]$}\in {\partial B_{r}(\hbox{\scriptsize$\bbm[\bar\xi]$})}}\Big(\sum_{j=1}^N\log K(\xi_j)\Big)
\\
&\stackrel{\eqref{Psi=D+log}}{ \leq }&
\min\limits_{\hbox{\scriptsize$\bbm[\xi]$}\in {\partial B_{r}(\hbox{\scriptsize$\bbm[\bar\xi]$})}}\Psi_{\underline{\alpha}}(\bbm[\xi]).
\end{eqnarray*}

\

\noindent STEP 2. {\it We show that 
 if \begin{equation}\label{hParz}
0<\sum_{i=1}^m \alpha_i  - 2N+\chi(\Sigma)<\frac{1}{4\pi}
\end{equation} 
then
    $\mathcal A (\bbm[\xi]^*_{\underline\alpha})>0$, where $\mathcal A$ is the function defined in \eqref{matA}.}
\\
\\
By definition of $\mathcal A$ it will be enough to prove that for any $j=1,\ldots,N$
\[
\Delta_{g}\log K(\xi_j^*)+\frac{8\pi N-4\pi\chi({\Sigma},\underline{\alpha})}{|{\Sigma}|}>0
\]
where $\bbm[\xi]^*_{\underline\alpha}=(\xi_1^*,\ldots,\xi_N^*).$
This holds true by \eqref{hParz} and the condition \eqref{terza}, since 
\begin{eqnarray*}
\Delta_{g}\log K(\xi_j^*)+\frac{8\pi N-4\pi\chi({\Sigma},\underline{\alpha})}{|{\Sigma}|} &= & \Delta_{g}\log K(\xi_j^*)+\frac{8\pi N-4\pi \left( \chi (\Sigma)+ \sum_{i=1}^m \alpha_i \right)}{|{\Sigma}|}
\\
&\stackrel{\eqref{hParz}}{ > }& \Delta_{g}\log K(\xi_j^*) - \frac{1}{|\Sigma|}
\stackrel{\eqref{terza}}{ \geq }0.
\end{eqnarray*}

\

\noindent STEP 3. {\it Conclusion. 
}
\\\\
By \eqref{alphatra},   Step 2 and Step 1 we get that  the assumptions a), b) and c) respectively of Proposition \ref{prop:espofigue} are all satisfied if \eqref{hParz} holds. 
As a consequence \ref{liouvhat} admits a solution for all $\rho\in (8\pi N, 8\pi N +\delta)$ under the assumption \eqref{hParz}, where  $\delta=\delta(\alpha_*,\alpha^*, K)\in (0,1)$ 
is provided in Proposition \ref{prop:espofigue}. 
The conclusion follows observing that  \eqref{hpsum}  implies \eqref{hParz} and $$\rho_{geo}= 4\pi\Big({\chi(\Sigma)}+\sum_{i=1}^m\alpha_i\Big) \in (8\pi N, 8\pi N +\delta).$$
\end{proof}

\

{Similarly one can define the class of functions
\begin{eqnarray}
\label{classeDiKmeno}
{\mathcal  K}^-_{\bar\xi,\alpha_*,\alpha^*}:=
\left\{K\in C^2(\Sigma)\ :\   {K} \mbox{ satisfies  \eqref{prima}, \eqref{seconda2}  and \eqref{terza2}}\right\}
\end{eqnarray}
where
\begin{equation}\label{seconda2}
\max_{\hbox{\scriptsize$\bbm[\xi]$}\in {\partial B_{r}(\hbox{\scriptsize$\bbm[\bar\xi]$})}} \Big(\frac{1}{4\pi}\sum_{j=1}^N \log K(\xi_j)\Big)
< 
\min_{\hbox{\scriptsize$\bbm[\xi]$}\in \overline{B_{\frac r2}(\hbox{\scriptsize$\bbm[\bar\xi]$})}}\Big(\frac{1}{4\pi}\sum_{j=1}^N \log K(\xi_j)\Big)+m^-_{\bar\xi,\alpha_*,\alpha^*}-M_{\bar\xi,\alpha_*,\alpha^*}^-
\end{equation}
for
\[
M^-_{\bar\xi,\alpha_*,\alpha^*}:=\!\!\max\limits_{{\tiny\begin{array}{c}\alpha_*\leq s_i\leq\alpha^*\\i=1,\ldots,m\end{array}}}\max\limits_{ 
\hbox{\scriptsize$\bbm[\xi]$}\in\partial B_r(\hbox{\scriptsize$\bbm[\bar\xi]$})    }D_{\underline{s}}(\bbm[\xi])
\qquad
\mbox{and}
\qquad 
m^-_{\bar\xi,\alpha_*,\alpha^*}:=\!\!\min\limits_{{\tiny\begin{array}{c}\alpha_*\leq s_i\leq\alpha^*\\i=1,\ldots,m\end{array}}}\min\limits_{\hbox{\scriptsize$\bbm[\xi]$}\in \overline{B_{\frac r2}(\hbox{\scriptsize$\bbm[\bar\xi]$})}} D_{\underline{s}}(\bbm[\xi]),
\]
and 
\begin{equation}\label{terza2}
\sup_{\hbox{\scriptsize$\bbm[\xi]$}\in {B_{r}(\hbox{\scriptsize$\bbm[\bar\xi]$})}}\max_{j=1,\ldots,N} \Delta_{g}\log K(\xi_j)\leq - \frac{1}{{|\Sigma|}},
\end{equation}
and prove the following result.}

\begin{theorem}\label{prop:example2} For any $K\in{\mathcal  K}^-_{{\bar\xi,\alpha_*,\alpha^*}}$ there exists $\delta=\delta(\alpha_*,\alpha^*,K)>0$ such that if
\begin{equation}\label{hpsum2} 
\sum_{i=1}^m \alpha_i  = 2N-\chi(\Sigma)-\frac{\varepsilon}{4\pi}\qquad  {\mbox{ for }\varepsilon\in (0, \delta)} 
\end{equation}
and moreover
\begin{equation}\label{alphatra2}\alpha_*\leq\alpha_i\leq\alpha^*\qquad\mbox{ for all $i=1,\ldots, m$}
\end{equation}
then \ref{liouvgeom} admits a solution {with $\rho_{geo}=8\pi N -\varepsilon$. } \end{theorem}

\begin{proof}
The proof can be obtained following the same strategy of the proof of Theorem \ref{prop:example1}.

{Indeed first we show that the assumptions \eqref{seconda2} and \eqref{alphatra2} imply the existence of a local maximum point $\bbm[\xi]^*_{\underline\alpha}$ for  $\Psi_{\underline{\alpha}}$ (instead of a local minimum point).
\\
Then we prove that $\mathcal A(\bbm[\xi]^*_{\underline\alpha})<0$  (instead of $\mathcal A(\bbm[\xi]^*_{\underline\alpha}) >0$) if $$0<2N-\chi(\Sigma)- \sum_{i=1}^m \alpha_i < \frac{1}{4\pi},$$ using the  assumptions \eqref{terza2}. 
\\
Last we apply again Proposition \ref{prop:espofigue} to deduce that \ref{liouvhat} admits a solution for all $\rho\in (8\pi N-\delta, 8\pi N)$, where  $\delta=\delta(\alpha_*,\alpha^*, K)\in (0,1)$ 
is as in Proposition \ref{prop:espofigue}. 
The conclusion follows observing that $$\rho_{geo}= 4\pi\Big({\chi(\Sigma)}+\sum_{i=1}^m\alpha_i\Big) \in (8\pi N-\delta, 8\pi N)$$ thanks to assumption \eqref{hpsum2}.
}
\end{proof}

\

We point out  that also in this case  for any fixed $\bbm[\bar\xi], \alpha_*,\alpha^*$, considering functions $K$ having \emph{sufficiently concave} local maxima at points $\bar\xi_1,\ldots,\bar\xi_N$, where $\bbm[\bar\xi]=(\bar\xi_1,\ldots,\bar\xi_N)$, we can find plenty of functions $K$,  both positive or sign-changing, in the class $ {\mathcal K}^-_{\bar\xi,\alpha_*,\alpha^*}$.

\

\subsection{Examples in the case of the standard sphere} Now we will easily deduce  Theorems \ref{thm:examples1}, \ref{thm:examples2}, \ref{thm:examples3} in the introduction by the above Theorems \ref{prop:example1}, \ref{prop:example2} if 
 $(\Sigma,g)=(\mathbb{S}^2,g_0)$.
 Indeed when $(\Sigma,g)=(\mathbb{S}^2,g_0)$ and  $m=1$,  by Theorems \ref{prop:example1}, \ref{prop:example2} we immediately get the following.
\begin{corollary}\label{cor:sferam=1}
Let $N\in\N$, $p_1\in\mathbb{S}^2$ and 
$\bar{\bbm[\xi]}\in(\mathbb{S}^2\setminus\{p_1\})^N\setminus\Delta$. 
Then for any $K\in\mathcal  K_{{\bar\xi,-\frac{1}{2},2N}}^{\pm}$ there exists $\delta=\delta(K)>0$ such that if
\begin{equation}\label{condAlphapm}
\alpha_1 =  2(N-1)\pm\frac{\varepsilon}{4\pi}\qquad  \mbox{ for }\varepsilon\in (0, \delta)
\end{equation}
then \ref{liouvgeom} admits a solution with $\rho_{geo}=8\pi N \pm \varepsilon$.
\end{corollary}

\begin{proof}[Proof of Theorem \ref{thm:examples1}] Theorem \ref{thm:examples1}  follows by Corollary \ref{cor:sferam=1} by taking $K$ in the subset of $\mathcal  K_{\bar{\xi},-\frac{1}{2},2N}^{+}$ made up of positive functions.
\end{proof}

\begin{remark} We emphasize  that Theorem \ref{thm:examples1} provides existence of a solution for \ref{liouvgeom} for special classes of functions $K$, whereas according to the result in \cite{Troy}   if $K\equiv 1$ then   \ref{liouvgeom} does not admit solutions on the standard sphere. 
\end{remark}

\begin{proof}[Proof of Theorem \ref{thm:examples2}]
Theorem \ref{thm:examples2}  follows by Corollary \ref{cor:sferam=1} by taking $\ell=0$ and considering  functions $K$ in the class $\mathcal  K_{{\bar{\xi},-\frac{1}{2},2N}}^{+}$ satisfying           (H1), (H2), (H3), (H4) with $(\mathbb{S}^2)^+$ contractible. It is easy to see that such functions exist (see Remark \ref{rimm} below).
\end{proof}
\begin{remark}\label{rimm}
Observe that it is possible to find examples of functions $K$ in the class considered in  Theorem \ref{thm:examples2} which are also axially symmetric.
For instance, if $m=1,$ $\ell=0$, $ N=1$, such an example is described by the first picture in   Figure \ref{figureUltima}: by locating  $p_1$ in the south pole and  $\bar\xi$ in the north pole we can construct an axially symmetric function $K$ with $(\mathbb{S}^2)^+$ coinciding with the upper hemisphere and having in $\bar\xi$ a sufficiently convex local minimum such that \ref{liouvgeom} admits a solution if $\alpha_1>0$ is sufficiently small. 

This is particularly interesting because in \cite{fra&raf} the authors exhibit   a class of axially symmetric  functions $K$ satisfying (H1), (H2), (H3), (H4) and with $(\mathbb{S}^2)^+$ contractible  for which \ref{liouvgeom} does not admit solutions if $m=1$, $\ell=0$  and $\alpha_1>0$ (see the second picture in Figure \ref{figureUltima}).  Of course this class of functions is different from the one considered  in Theorem \ref{thm:examples2}. \end{remark}

\bigskip

\begin{figure}[h]
  \centering
  \def\svgwidth{220pt}
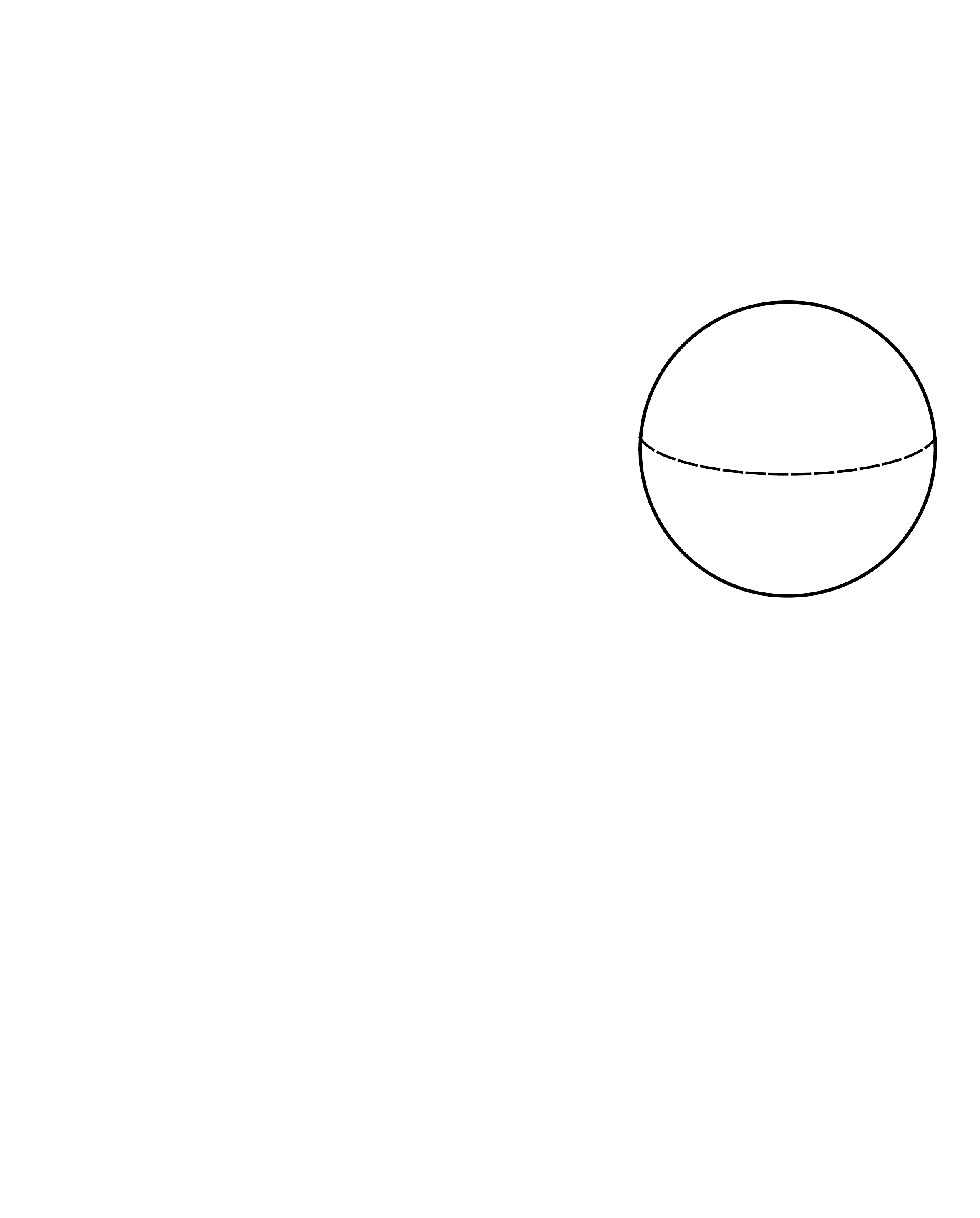
 \vspace{-1.0cm} \caption{$(\mathbb{S}^2,g_0)$, $\Sigma^+$ contractible, $\rho_{geo}= 8\pi+\varepsilon$}
\label{figureUltima}
\end{figure}

\bigskip

Again for $(\Sigma,g)=(\mathbb{S}^2,g_0)$, when $m=3$ and $N=2$ we obtain  the following special case from Theorem  \ref{prop:example1}. \begin{corollary}\label{cor:ultimo}
Let  $p_1,p_2,p_3\in \mathbb S^2$,
$\bar{\bbm[\xi]}\in(\mathbb{S}^2\setminus\{p_1,p_2,p_3\})^2\setminus\Delta$ and  
$-1<\alpha_{\star}\leq \alpha^{\star}$.
Then  for any $K\in\mathcal  K_{\bar\xi,\alpha_*,\alpha^*}^+$ there exists $\delta=\delta(\alpha_*,\alpha^*,K)>0$ such that if
\begin{equation}
\sum_{i=1}^3 \alpha_i  =  2+\frac{\varepsilon}{4\pi}\qquad  \mbox{ for }\varepsilon\in (0, \delta)
\end{equation}
and moreover
\begin{equation}\alpha_*\leq\alpha_i\leq\alpha^*\qquad\mbox{ for all $i=1,2,3$}
\end{equation}
then \ref{liouvgeom} admits a solution with $\rho_{geo}=16\pi  +\varepsilon$.
\end{corollary}

\

{
\begin{proof}[Proof of Theorem \ref{thm:examples3}]
We just apply Corollary \ref{cor:ultimo} for $\alpha_*=-\frac{1}{2}$,  $\alpha^*=3$ and fixing 
 $\alpha_1=\alpha_2\in(-\tfrac13,0)$, so that $\alpha_3= 2-2\alpha_1+\frac{\varepsilon}{4\pi}$. 
\end{proof}
\begin{remark}
We emphasize that  Theorem \ref{thm:examples3} assures the existence of a solution for \ref{liouvgeom} in a case when,
 if $K$ is also positive, the Leray-Schauder degree of the equation \ref{liouvgeom} vanishes according to the formula in \cite{ChenLin}.
\end{remark}
}

\subsection*{Acknowledgements} The first author has been supported by the PRIN-project 201274FYK7\_007.

 The second author has been supported by the PRIN-project $201274$FYK7$\_005$ and by Fondi Avvio alla Ricerca Sapienza 2015. 
 
 The third author has been supported by the PRIN-project 201274FYK7\_005.

\end{document}